\documentclass[11pt,letterpaper]{amsart}
\usepackage{amsmath, amssymb, stmaryrd, mathtools}
\usepackage[foot]{amsaddr}
\usepackage[margin=1in]{geometry}

\title[Ferrari--Spohn diffusions to Airy line ensemble]{Uniform convergence of Dyson Ferrari--Spohn diffusions to the Airy line ensemble}

\author{Evgeni Dimitrov$^\ast$}
\author{Christian Serio$^\dagger$}

\address{$^\ast$Department of Mathematics, University of Southern California, Los Angeles, CA 90007, USA}
\email{$^\ast$edimitro@usc.edu}

\address{$^\dagger$Department of Mathematics, Stanford University, Stanford, CA 94305, USA}
\email{$^\dagger$cdserio@stanford.edu}

\newcommand{\weyl}{W^\circ}

\newtheorem{theorem}{Theorem}
\newtheorem{proposition}{Proposition}[section]
\newtheorem{lemma}{Lemma}[section]

\theoremstyle{definition}
\newtheorem{definition}{Definition}[section]
\theoremstyle{remark}
\newtheorem{remark}{Remark}[section]
\usepackage{verbatim}
\usepackage{hyperref}
\hypersetup{
	colorlinks   = true,
	urlcolor     = black,
	linkcolor    = blue,
	citecolor   = blue
}
\begin{document}
	
	\maketitle
	
	\begin{abstract}
		We consider the Dyson Ferrari--Spohn diffusion $\mathcal{X}^N = (\mathcal{X}^N_1,\dots,\mathcal{X}^N_N)$, consisting of $N$ non-intersecting Ferrari--Spohn diffusions $\mathcal{X}^N_1 > \cdots > \mathcal{X}^N_N > 0$ on $\mathbb{R}$. This object was introduced by Ioffe, Velenik, and Wachtel (2018) as a scaling limit for line ensembles of $N$ non-intersecting random walks above a hard wall with area tilts, which model certain three-dimensional interfaces in statistical physics. It was shown by Ferrari and Shlosman (2023) that as $N\to\infty$, after a spatial shift of order $N^{2/3}$ and constant rescaling in time, the top curve $\mathcal{X}^N_1$ converges to the $\mathrm{Airy}_2$ process in the sense of finite-dimensional distributions. We extend this result by showing that the full ensemble $\mathcal{X}^N$ converges with the same shift and time scaling to the Airy line ensemble in the topology of uniform convergence on compact sets. In our argument we formulate a Brownian Gibbs property with area tilts for $\mathcal{X}^N$, which we show is equivalent after a global parabolic shift to the usual Brownian Gibbs property introduced by Corwin and Hammond (2014).
	\end{abstract}
	
	\tableofcontents

	\section{Introduction and main result}\label{sec1}
	
	\subsection{Background and motivation}\label{sec1.1} In this paper, we consider scaling limits of a certain class of Gibbsian line ensembles consisting of \textit{non-intersecting random walks under area tilts}, which act as effective models for low-temperature interfaces in three-dimensional lattice models in statistical physics. The basic model consists of $N$ mean-zero discrete random walk trajectories $X_1,\dots,X_N$ on an integer interval $\llbracket -M,M\rrbracket = [-M,M]\cap\mathbb{Z}$. The walks are conditioned to remain ordered and stay above a hard wall at zero, i.e., $X_1(j) > \cdots > X_N(j) > 0$ for all $j\in\llbracket -M,M\rrbracket$. Furthermore, each walk is \textit{tilted} by a factor representing the area under it, that is, the probability of the trajectories is weighted by the additional factor
	\begin{equation}\label{geomtilt}
		\exp\left(-\mathfrak{a}\sum_{i=1}^N \mathfrak{b}^{i-1}\frac{1}{M}\sum_{j=-M}^{M-1} X_i(j)\right).
	\end{equation}
	Here $\mathfrak{a}>0$ and $\mathfrak{b}\geq 1$ are fixed parameters. This model was introduced in \cite{IVW} with $\mathfrak{b}=1$ in order to mimic level lines of $(2+1)$-dimensional interfaces, and later developed in \cite{IV, CIW1} with $\mathfrak{b}>1$ to describe in particular fluctuations of level lines of the ($2+1$)D solid-on-solid (SOS) model above a wall on a growing box. A broad discussion of approximations of the above type for level lines of SOS models can be found in the survey \cite[Section 4]{IV}. For further details on the particular case $\mathfrak{b}>1$ and its connections with level lines of the SOS model above a wall, we refer to \cite{SOS2,CIW1,CIW2,DLZ}, and \cite[Section 1.1]{Ser2} for an overview. The case $\mathfrak{b}=1$ mimics level lines of an SOS model with bulk Bernoulli fields and no wall constraint describing crystal growth; see \cite{IV, IS}, as well as \cite{FS} for a discussion in the context of interfaces in the 3D Ising model.

	The number of level lines of the height function grows with the size of the box, so it is natural to consider scaling limits of the random walk model as $M,N\to\infty$. The limiting behavior is notably different depending on whether $\mathfrak{b}=1$ or $\mathfrak{b}>1$. When $\mathfrak{b}>1$, the area tilts are geometrically growing, so that lower curves are pushed downwards more strongly. It has been shown that under diffusive $1$:$2$:$3$ scaling, i.e., rescaling in time by $M^{2/3}$ and in space by $M^{-1/3}$, with suitable boundary conditions the line ensemble converges uniformly on compact sets to a unique limit as $M\to\infty$ first and then $N\to\infty$; see \cite{CIW1, CIW2, DLZ, Ser2}. We note that in this case the scaling need not depend on $N$. The recent work \cite{CG} shows that for the analogous model with random walks replaced by Brownian bridges with no rescaling, this convergence holds for general boundary conditions with $M,N\to\infty$ in any order (this was previously known for zero boundary conditions due to \cite{CIW2}). Due to a lack of integrable structure, little is known about the nature of the scaling limit; mixing (hence ergodicity) and tail estimates are established in \cite{CG}, but a precise description remains open.
	
	In this paper we focus on the scaling limit of the model with $\mathfrak{b}=1$. This model was first analyzed in \cite{ISV} with a single curve, and subsequently in \cite{IVW} with the number of curves $N\geq 1$ fixed. In the former paper it was shown that after the same rescaling described above, a random walk bridge with area tilt converges to an explicit limit known as the \textit{Ferrari--Spohn diffusion}, assuming $O(M^{1/3})$ boundary conditions. The latter paper extended this convergence result to ensembles of $N$ bridges with area tilts, and constructed a limit $\mathcal{X}^N = (\mathcal{X}^N_1,\dots,\mathcal{X}^N_N)$ known as the \textit{Dyson Ferrari--Spohn diffusion}. We will describe these objects in more detail in Section \ref{sec1.2}. For a discussion of the broader role of Ferrari--Spohn diffusions as scaling limits for various interface models in statistical mechanics, we refer to \cite{IV}. 
	
	As the number of curves $N\to\infty$, the top curve $\mathcal{X}^N_1$ of the Dyson Ferrari--Spohn diffusion escapes to infinity and the line ensemble is not tight, unlike in the case $\mathfrak{b}>1$. However, it was shown in \cite{FS} using a determinantal point process analysis that after subtracting an appropriate correction depending on $N$ and rescaling in time, the top curve converges to the $\mathrm{Airy}_2$ process. Our main result, which we state in Section \ref{sec1.3}, is to strengthen this to uniform convergence of the full line ensemble $\mathcal{X}^N$ to the Airy line ensemble.
	
	\subsection{Dyson Ferrari--Spohn diffusions and the Airy line ensemble}\label{sec1.2}
	In this section we introduce the Dyson Ferrari--Spohn diffusions and the Airy line ensemble.
	
	The \textit{Ferrari--Spohn diffusion} was first introduced in \cite{FSpohn}, where it was derived as a scaling limit for Brownian motion constrained to stay above a circular or parabolic barrier. It is the stationary diffusion process on $\mathbb{R}$ whose fixed-time distribution has density proportional to $\mathrm{Ai}(x-\omega_1)^2 \cdot {\bf 1}\{x > 0\}$, and whose infinitesimal generator is given by
	\begin{equation}
		L = \frac{1}{2}\frac{\mathrm{d}^2}{\mathrm{d}x^2} + a(x)\frac{\mathrm{d}}{\mathrm{d}x},
	\end{equation}
	where $a(x) = \frac{\mathrm{d}}{\mathrm{d}x}\log \mathrm{Ai}(x-\omega_1)$. Here, $\mathrm{Ai}$ denotes the Airy function, and $-\omega_1$ its largest zero.
	
	The \textit{Dyson Ferrari--Spohn diffusion} $\mathcal{X}^N$ with $N$ curves was introduced in \cite{IVW} as a collection of $N$ non-intersecting Ferrari--Spohn diffusions. More precisely, $\mathcal{X}^N = (\mathcal{X}_1^N,\dots,\mathcal{X}_N^N)$ is the stationary diffusion on $\mathbb{R}$, whose fixed-time distribution has density proportional to $\Delta(\vec{x})^2 \cdot {\bf 1}\{x_1 > \cdots > x_N > 0\}$, and whose infinitesimal generator is given by
	\begin{equation}\label{S1GEN}
		L_N = \sum_{k=1}^N \left(\frac{1}{2}\frac{\partial^2}{\partial x_k^2} + \frac{\partial\log \Delta(\vec{x})}{\partial x_k} \cdot\frac{\partial}{\partial x_k}\right),
	\end{equation}
	where $\Delta(\vec{x}) = \det[\mathrm{Ai}( x_{N-j+1} - \omega_i)]_{i,j=1}^N$ and  $-\omega_1 > -\omega_2 > \cdots$ are the zeros of the Airy function $\mathrm{Ai}$. We mention here that our definition of $\mathcal{X}^N$ differs slightly from \cite{FS} as our indexing of the curves is reversed. Further descriptions of this process can be found in \cite[Section 2]{IVW} and \cite[Section 3]{FS}, and in Section \ref{sec4.1} below we will specifically utilize the description of $\mathcal{X}^N$ in terms of determinantal point processes from \cite[Section 3.3]{FS}.\\

	The {\em Airy line ensemble} $\mathcal{A}^{\mathrm{Airy}} = \{\mathcal{A}_i^{\mathrm{Airy}}\}_{i = 1}^\infty$ is a collection of infinitely many random continuous curves on $\mathbb{R}$, which was first introduced in \cite{Spohn} and later extensively studied in \cite{CH14}. The Airy line ensemble is a universal scaling limit for various models that belong to the {\em KPZ universality class}, see \cite{CU2} for an expository review of this class.  In \cite{Spohn} the convergence to the Airy line ensemble (in the finite-dimensional sense) was established for the polynuclear growth model and in \cite{CH14} it was shown for Dyson Brownian motion (in a stronger uniform sense). Recently, \cite{DNV} established the uniform convergence of various classical integrable models to the Airy line ensemble including non-colliding random walks with Bernoulli, geometric and Poisson jumps and Brownian last passage percolation. We refer to the introduction of \cite{DNV} for a more extensive discussion of the history, motivation behind and progress on the problem of establishing convergence to the Airy line ensemble. The {\em parabolic Airy line ensemble} is a collection of infinitely many random continuous curves $\mathcal{L}^{\mathrm{Airy}} = \{\mathcal{L}^{\mathrm{Airy}}_i\}_{i = 1}^\infty$ on $\mathbb{R}$ and is related to the Airy line ensemble via $\mathcal{A}^{\mathrm{Airy}}_i(t) = 2^{1/2} \mathcal{L}^{\mathrm{Airy}}_i(t) + t^2$ for $i \in \mathbb{N}$. We mention that \cite{DNV} uses a slightly different notation than what we use in the present paper\textemdash see the discussion at the end of \cite[Section 1.1]{DM}.
	
	There are various ways one can define the parabolic Airy line ensemble $\mathcal{L}^{\mathrm{Airy}}$. A description in terms of determinantal processes can be found in \cite[Definition 2.1]{DNV}. In \cite{CH14} the authors defined $\mathcal{L}^{\mathrm{Airy}}$ as the (\hspace{-.3mm}{\em unique}) edge scaling limit of $N$ non-intersecting Brownian bridges, and they showed that this ensemble enjoys a {\em Brownian Gibbs property}. For a formal definition we refer the reader to \cite[Section 2]{CH14} and \cite[Section 2]{DM}, but in words this property states that locally the structure of the curves in $\mathcal{L}^{\mathrm{Airy}}$ is that of independent Brownian bridges that have been conditioned to avoid each other. From \cite[Theorem 1.1]{DM} we also have that $\mathcal{L}^{\mathrm{Airy}}$ is the unique line ensemble such that: (1) $2^{1/2} \mathcal{L}^{\mathrm{Airy}}_1(t) + t^2$ is equal in distribution to the {\em Airy\hspace{.3mm}$_{2}$ process} from \cite{Spohn}; (2) $\mathcal{L}^{\mathrm{Airy}}$ satisfies the Brownian Gibbs property.
	
	\subsection{Main result}\label{sec1.3}
	
	Our result concerns the scaling limit of $\mathcal{X}^N$ as $N\to\infty$. The first result in this direction was \cite[Theorem 1.1]{FS}, which shows that $\mathcal{X}_1^N(2t) - c_1 N^{2/3}$ converges in finite-dimensional distributions to the $\mathrm{Airy}_2$ process as $N\to\infty$, where $c_1 = (3\pi/2)^{2/3}$. Motivated by this result, we now state our main theorem extending this convergence to the full line ensemble.
	
	\begin{theorem}\label{main}
		Let $\mathcal{X}^N$ denote the Dyson Ferrari--Spohn diffusion with $N$ curves, and define the shifted and time-rescaled line ensemble $\widetilde{\mathcal{X}}^N$ by 
		\begin{equation}\label{FSrescaled}
			\widetilde{\mathcal{X}}_i^N(t) = \mathcal{X}_i^N(2t) - c_1 N^{2/3}, \qquad 1\leq i\leq N,\quad t\in\mathbb{R},
		\end{equation}
		where $c_1 = (3\pi/2)^{2/3}$. Then, $\widetilde{\mathcal{X}}^N$ converges weakly as $N\to\infty$ to the Airy line ensemble $\mathcal{A}^{\mathrm{Airy}}$, in the topology of uniform convergence on compact sets on $C(\mathbb{N}\times\mathbb{R})$.
	\end{theorem}
	Our proof of Theorem \ref{main} goes through the following steps. Firstly, we use the determinantal point process description of $\mathcal{X}^N$ from \cite[Section 3]{FS} and show that the estimates on the correlation kernel from that paper are sufficient to prove that $\mathcal{X}^N$ converges in the finite-dimensional sense to $\mathcal{A}^{\mathrm{Airy}}$. This step is a straightforward adaptation of the argument in \cite[Proposition 3.1]{DNV} and its proof is presented in Section \ref{sec4.1}. In order to improve our finite-dimensional convergence to uniform over compacts convergence we seek to apply \cite[Theorem 3.8]{CH14}; however, the immediate obstacle we face is that $\mathcal{X}^N$ does not satisfy the Brownian Gibbs property. Rather, $\mathcal{X}^N$ satisfies a certain area-tilted Brownian Gibbs property. A formal definition of the latter can be found in Definition \ref{DefPBGP}, and the fact that $\mathcal{X}^N$ satisfies such a property is deduced in Section \ref{sec3.2}. Here, we essentially rely on \cite[Theorem A]{IVW}, which shows that $\mathcal{X}^N$ is the scaling limit of area-tilted avoiding random walks, and we show that this area-tilted Gibbs property survives the limit with random walks becoming Brownian motions. Finally, to set ourselves up for the application of \cite[Theorem 3.8]{CH14}, we show that the if one subtracts an appropriate parabola from $\mathcal{X}^N$, then the resulting ensemble satisfies the usual Brownian Gibbs property from \cite{CH14}. This result, formulated generally in Section \ref{sec2.2}, is a consequence of the Cameron-Martin-Girsanov theorem.
	
	While Theorem \ref{main} is the main result of the paper, we hope that Section \ref{sec2} will serve as a useful reference on area-tilted Brownian Gibbsian line ensembles for others working in the area.\\
	
	\noindent {\bf Acknowledgments.} The authors would like to thank Ivan Corwin for suggesting the use of the Cameron-Martin-Girsanov theorem in the proof of Lemma \ref{GibbsPar}, Amir Dembo for helpful discussions and advice, and Senya Shlosman for bringing the paper \cite{FS} to their attention. E.D. is partially supported by NSF grant DMS:2054703.
	
	%----------------------------------------------------------------------------------------------------------------------------------------------------------------------------------------------------------------------------
	%
	%     Section 2
	%
	%-------------------------------------------------------------------------------------------------------------------------------------------------------------------------------------------------------------------------
	\section{Line ensembles and area tilts}\label{sec2}
	
	In this section we introduce some basic definitions and technical results used later in the paper.
	
	%----------------------------------------------------------------------------------------------------------------------------------------------------------------------------------------------------------------------------
	%
	%     Section 2.1
	%
	%----------------------------------------------------------------------------------------------------------------------------------------------------------------------------------------------------------------------------
	\subsection{Line ensembles and Brownian Gibbs properties}\label{sec2.1}
	We begin by introducing the notions of a {\em line ensemble} and the (\textit{partial$\,)$ Brownian Gibbs property with respect to $\vec{A}$-area tilts}. Our exposition in this section closely follows that of  \cite[Section 2]{CH16} and \cite[Section 2]{DM}. 
	
	Given two integers $p \leq q$, we let $\llbracket p, q \rrbracket$ denote the set $\{p, p+1, \dots, q\}$. Given an interval $\Lambda \subset \mathbb{R}$, we endow it with the subspace topology of the usual topology on $\mathbb{R}$. We let $(C(\Lambda), \mathcal{C})$ denote the space of continuous functions $f: \Lambda \rightarrow \mathbb{R}$ with the topology of uniform convergence over compacts, see \cite[Chapter 7, Section 46]{Munkres}, and Borel $\sigma$-algebra $\mathcal{C}$. Given a set $\Sigma \subset \mathbb{Z}$, we endow it with the discrete topology and denote by $\Sigma \times \Lambda$ the set of all pairs $(i,x)$ with $i \in \Sigma$ and $x \in \Lambda$ with the product topology. We also denote by $\left(C (\Sigma \times \Lambda), \mathcal{C}_{\Sigma}\right)$ the space of continuous functions on $\Sigma \times \Lambda$ with the topology of uniform convergence over compact sets and Borel $\sigma$-algebra $\mathcal{C}_{\Sigma}$. We will typically take $\Sigma = \llbracket 1, N \rrbracket$ (we use the convention $\Sigma = \mathbb{N}$ if $N = \infty$).
	
	The following defines the notion of a line ensemble.
	\begin{definition}\label{DefLE}
		Let $\Sigma \subset \mathbb{Z}$ and $\Lambda \subset \mathbb{R}$ be an interval. A {\em $\Sigma$-indexed line ensemble $\mathcal{L}$} is a random variable defined on a probability space $(\Omega, \mathcal{F}, \mathbb{P})$ that takes values in $\left(C (\Sigma \times \Lambda), \mathcal{C}_{\Sigma}\right)$. Intuitively, $\mathcal{L}$ is a collection of random continuous curves (sometimes referred to as {\em lines}), indexed by $\Sigma$,  each of which maps $\Lambda$ in $\mathbb{R}$. We will often slightly abuse notation and write $\mathcal{L}: \Sigma \times \Lambda \rightarrow \mathbb{R}$, even though it is not $\mathcal{L}$ which is such a function, but $\mathcal{L}(\omega)$ for every $\omega \in \Omega$. For $i \in \Sigma$ we write $\mathcal{L}_i(\omega) = (\mathcal{L}(\omega))(i, \cdot)$ for the curve of index $i$ and note that the latter is a map $\mathcal{L}_i: \Omega \rightarrow C(\Lambda)$, which is $(\mathcal{C}, \mathcal{F})-$measurable. We call a line ensemble {\em non-intersecting} if $\mathbb{P}$-almost surely $\mathcal{L}_i(r) > \mathcal{L}_j(r)$  for all $i < j$ and $r \in \Lambda$.
	\end{definition}
	
	We next turn to formulating the partial Brownian Gibbs property with respect to $\vec{A}$-area tilts\textemdash we do this in Definition \ref{DefPBGP} after introducing some relevant notation. If $W_t$ denotes a standard one-dimensional Brownian motion, then the process
	$$\tilde{B}(t) =  W_t - t W_1, \hspace{5mm} 0 \leq t \leq 1,$$
	is called a {\em Brownian bridge $($from $\tilde{B}(0) = 0$ to $\tilde{B}(1) = 0)$ with diffusion parameter $1$.} 
	
	Given $a, b,x,y \in \mathbb{R}$ with $a < b$, we define a random variable on $(C([a,b]), \mathcal{C})$ through
	\begin{equation}\label{BBDef}
		B(t) = (b-a)^{1/2}\, \tilde{B} \left( \frac{t - a}{b-a} \right) + \frac{b-t}{b-a} \cdot x + \frac{t- a}{b-a}\cdot y, 
	\end{equation}
	and refer to the law of this random variable as a {\em Brownian bridge $($from $B(a) = x$ to $B(b) = y)$ with diffusion parameter $1$.} Given $k_1, k_2 \in \mathbb{Z}$ with $k_1 \leq k_2$ and $\vec{x}, \vec{y} \in \mathbb{R}^{k_2 - k_1 + 1}$, we let $\mathbb{P}^{k_1, k_2, a,b, \vec{x},\vec{y}}_{\mathrm{free}}$ denote the law of $k_2 - k_1 + 1$ independent Brownian bridges $\{B_i: [a,b] \rightarrow \mathbb{R} \}_{i = k_1}^{k_2}$ from $B_i(a) = x_{i-k_1 +1}$ to $B_i(b) = y_{i - k_1 + 1}$ all with diffusion parameter $1$. We will denote the expectation with respect to $\mathbb{P}^{k_1, k_2, a,b, \vec{x},\vec{y}}_{\mathrm{free}}$ by $\mathbb{E}^{k_1, k_2, a,b, \vec{x},\vec{y}}_{\mathrm{free}}.$ If $k_1=k_2=1$, we omit the first two superscripts.
	
	The following definition introduces the notion of an $(f,g)$-avoiding Brownian line ensemble, which in plain words can be understood as a random ensemble of $k$ independent Brownian bridges, conditioned on not crossing each other and staying above the graph of $g$ and below the graph of $f$ for two continuous functions $f$ and $g$.
	\begin{definition}\label{DefAvoidingLaw}
		Let $k_1, k_2 \in \mathbb{Z}$ with $k_1 \leq k_2$ and put $k = k_2 - k_1 + 1$.  Let $\weyl_k$ denote the open Weyl chamber in $\mathbb{R}^k$, i.e.
		$$\weyl_k = \{ \vec{x} = (x_1, \dots, x_k) \in \mathbb{R}^k: x_1 > x_2 > \cdots > x_k \}.$$
		Let $\vec{x}, \vec{y} \in \weyl_k$, $a,b \in \mathbb{R}$ with $a < b$, and $f: [a,b] \rightarrow (-\infty, \infty]$ and $g: [a,b] \rightarrow [-\infty, \infty)$ be two continuous functions. The latter condition means that either $f: [a,b] \rightarrow \mathbb{R}$ is continuous or $f = \infty$ everywhere, and similarly for $g$. We also assume that $f(t) > g(t)$ for all $t \in[a,b]$, $f(a) > x_1, f(b) > y_1$ and $g(a) < x_k, g(b) < y_k.$
		
		With the above data we define the {\em $(f,g)$-avoiding Brownian line ensemble on the interval $[a,b]$ with entrance data $\vec{x}$ and exit data $\vec{y}$} to be the $\Sigma$-indexed line ensemble $\mathcal{Q}$ with $\Sigma = \llbracket k_1, k_2\rrbracket$ on $\Lambda = [a,b]$ and with the law of $\mathcal{Q}$ equal to $\mathbb{P}^{k_1, k_2, a,b, \vec{x},\vec{y}}_{\mathrm{free}}$ (the law of $k$ independent Brownian bridges $\{B_i: [a,b] \rightarrow \mathbb{R} \}_{i = k_1}^{k_2}$ from $B_i(a) = x_{i - k_1+ 1}$ to $B_i(b) = y_{i-k_1 + 1}$), conditioned on the event 
		$$\left\{ f(r) > B_{k_1}(r) > B_{k_1+1}(r) > \cdots > B_{k_2}(r) > g(r) \mbox{ for all $r \in[a,b]$} \right\}.$$ 
		We refer the interested reader to \cite[Definition 2.4]{DM} for more on why the above conditioning is justified. We denote the probability distribution of $\mathcal{Q}$ as $\mathbb{P}_{\mathrm{avoid}}^{k_1, k_2, a,b, \vec{x}, \vec{y}, f, g}$ and write $\mathbb{E}_{\mathrm{avoid}}^{k_1, k_2, a,b, \vec{x}, \vec{y}, f, g}$ for the expectation with respect to this measure. In particular, we have that 
		\begin{equation}\label{freeToAvoid}
			\frac{\mathrm{d} \mathbb{P}_{\mathrm{avoid}}^{k_1, k_2, a,b, \vec{x}, \vec{y}, f, g}}{\mathrm{d} \mathbb{P}_{\mathrm{free}}^{k_1, k_2, a,b, \vec{x}, \vec{y}}} = \frac{{\bf 1}  \left\{  f(r) > B_{k_1}(r)  > \cdots > B_{k_2}(r) > g(r)  \mbox{ for all $r \in[a,b]$} \right\}}{ \mathbb{P}_{\mathrm{free}}^{k_1, k_2, a,b, \vec{x}, \vec{y}} \left(\left\{  f(r) > B_{k_1}(r) > \cdots > B_{k_2}(r) > g(r)  \mbox{ for all $r \in[a,b]$}  \right\} \right)}.
		\end{equation}
	\end{definition}
	
	The following definition introduces the notion of an $(f,g)$-avoiding Brownian line ensemble with $\vec{A}$-area tilts, which in plain words can be understood as an $(f,g)$-avoiding Brownian line ensemble, whose $i$-th curve is reweighed by the area of the region enclosed by it and $g$. 
	\begin{definition}\label{DefWeightedLaw}
		Let $k_1, k_2, k, a, b, \vec{x}, \vec{y},f,g$ be as in Definition \ref{DefAvoidingLaw} and $g \neq -\infty$. Suppose that $A_{i} \in [0, \infty)$ for $i\in \llbracket k_1,k_2\rrbracket $ and $\vec{A} = (A_{k_1}, \dots, A_{k_2})$. 
		
		With the above data we define the {\em $(f,g)$-avoiding Brownian line ensemble with $\vec{A}$-area tilts on the interval $[a,b]$ with entrance data $\vec{x}$ and exit data $\vec{y}$} to be the $\Sigma$-indexed line ensemble $\mathcal{Q}$ with $\Sigma = \llbracket k_1, k_2\rrbracket$ on $\Lambda = [a,b] $ whose law $\mathbb{P}_{\mathrm{avoid}, \vec{A}}^{k_1, k_2, a,b, \vec{x}, \vec{y}, f,g}$ is given in terms of the  Radon-Nikodym derivative
		\begin{equation}\label{weightedToAvoid}
			\frac{\mathrm{d} \mathbb{P}_{\mathrm{avoid}, \vec{A}}^{k_1, k_2, a,b, \vec{x}, \vec{y}, f,g}}{\mathrm{d}\mathbb{P}^{k_1, k_2, a,b, \vec{x},\vec{y},f,g}_{\mathrm{avoid}}} (\mathcal{Q}_{k_1}, \dots, \mathcal{Q}_{k_2}) =  \frac{1}{Z_{\mathrm{avoid}, \vec{A}}^{k_1, k_2, a,b, \vec{x}, \vec{y}, f,g}}\exp \left( - \sum_{i = k_1}^{k_2} \int_a^b A_i \cdot  \left( \mathcal{Q}_{i}(u) - g(u)\right)\mathrm{d}u \right).
		\end{equation}
		In the last equation, the {\em normalization constant} is given by
		\begin{equation}\label{S2E1}
			Z_{\mathrm{avoid}, \vec{A}}^{k_1, k_2, a,b, \vec{x}, \vec{y}, f,g} = \mathbb{E}^{k_1, k_2, a,b, \vec{x},\vec{y},f,g}_{\mathrm{avoid}}\left[ \exp \left( - \sum_{i = k_1}^{k_2} \int_a^b A_i \cdot \left( \mathcal{Q}_{i}(u) - g(u)\right)\mathrm{d}u \right) \right].
		\end{equation}
		where on the right $(\mathcal{Q}_{k_1}, \dots, \mathcal{Q}_{k_2})$ is distributed according to $\mathbb{P}^{k_1, k_2, a,b, \vec{x},\vec{y},f,g}_{\mathrm{avoid}}$. Notice that since $A_{i} \in [0,\infty)$ for $i\in\llbracket k_1,k_2\rrbracket $, and under $\mathbb{P}^{k_1, k_2, a,b, \vec{x},\vec{y}}_{\mathrm{avoid}}$ the curves $\mathcal{Q}_{k_1}, \dots, \mathcal{Q}_{k_2}$ are all above $g$, we have that the right side of (\ref{weightedToAvoid}) is a continuous function on $C(\Sigma \times \Lambda)$, taking values in $(0,1]$. In particular, $Z_{\mathrm{avoid}, \vec{A}}^{k_1, k_2, a,b, \vec{x}, \vec{y}, f,g}  \in (0,1]$ and the above Radon-Nikodym derivative, and consequently $\mathbb{P}_{\mathrm{avoid}, \vec{A}}^{k_1, k_2, a,b, \vec{x}, \vec{y}, f,g}$, are well-defined. We write $\mathbb{E}_{\mathrm{avoid}, \vec{A}}^{k_1, k_2, a,b, \vec{x}, \vec{y}, f,g}$ for the expectation with respect to $\mathbb{P}_{\mathrm{avoid}, \vec{A}}^{k_1, k_2, a,b, \vec{x}, \vec{y}, f,g}$.
	\end{definition}

	We finally introduce the partial Brownian Gibbs property with respect to $\vec{A}$-area tilts.
	\begin{definition}\label{DefPBGP}
		Fix a set $\Sigma = \llbracket 1 , N \rrbracket$ with $N \in \mathbb{N}$ or $N  = \infty$ and an interval $\Lambda \subset \mathbb{R}$. In addition, let $A_i \in [0, \infty)$ for $i \in \Sigma$ and put $\vec{A} = \{A_i\}_{i \in \Sigma}$.
		
		A $\Sigma$-indexed line ensemble $\mathcal{L}$ on $\Lambda$ is said to satisfy the {\em partial Brownian Gibbs property with respect to $\vec{A}$-area tilts} if and only if it is non-intersecting and for any finite $K = \{k_1, k_1 + 1, \dots, k_2 \} \subset \Sigma$ with $k_2 \leq N - 1$ (if $\Sigma \neq \mathbb{N}$), $[a,b] \subset \Lambda$ and any bounded Borel-measurable function $F: C(K \times [a,b]) \rightarrow \mathbb{R}$ we have $\mathbb{P}$-almost surely
		\begin{equation}\label{PBGPTower}
			\mathbb{E} \left[ F(\mathcal{L}|_{K \times [a,b]}) {\big \vert} \mathcal{F}_{\mathrm{ext}} (K \times (a,b))  \right] =\mathbb{E}_{\mathrm{avoid}, \vec{A}_K}^{k_1, k_2, a,b, \vec{x}, \vec{y}, f, g} \bigl[ F({\mathcal{Q}}) \bigr],
		\end{equation}
		where $\vec{A}_K = (A_{k_1}, \dots, A_{k_2})$, $D_{K,a,b} = K \times (a,b)$ and $D_{K,a,b}^c = (\Sigma \times \Lambda) \setminus D_{K,a,b}$,
		$$\mathcal{F}_{\mathrm{ext}} (K \times (a,b)) = \sigma \left \{ \mathcal{L}_i(s): (i,s) \in D_{K,a,b}^c \right\}$$
		is the $\sigma$-algebra generated by the variables in the brackets above, and $ \mathcal{L}|_{K \times [a,b]}$ denotes the restriction of $\mathcal{L}$ to the set $K \times [a,b]$. On the right side of (\ref{PBGPTower}),  $\mathcal{Q} = (\mathcal{Q}_{k_1}, \dots, \mathcal{Q}_{k_2})$ is the $\llbracket k_1, k_2 \rrbracket \times (a,b)$-indexed $(f,g)$-avoiding Brownian line ensemble with $\vec{A}_K$-area tilts  with entrance data $\vec{x}$, exit data $\vec{y}$ and boundary data $(f,g)$ as in Definition \ref{DefWeightedLaw}. Here $\vec{x} = (\mathcal{L}_{k_1}(a), \dots, \mathcal{L}_{k_2}(a))$, $\vec{y} = (\mathcal{L}_{k_1}(b), \dots, \mathcal{L}_{k_2}(b))$, $f = \mathcal{L}_{k_1 - 1}[a,b]$ (the restriction of $\mathcal{L}$ to the set $\{k_1 - 1 \} \times [a,b]$) with the convention that $f = \infty$ if $k_1 - 1 \not \in \Sigma$, and $g = \mathcal{L}_{k_2 +1}[a,b]$. 
		
		If we furthermore have that $N = \infty$ we say that $\mathcal{L}$ satisfies the {\em Brownian Gibbs property with respect to $\vec{A}$-area tilts} (i.e. we drop ``partial'').
	\end{definition}
	
	\begin{remark} Observe that if $N = 1$ then the conditions in Definition \ref{DefPBGP} become void, i.e., any line ensemble with one line satisfies the partial Brownian Gibbs property with respect to $\vec{A}$-area tilts.
	\end{remark}
	\begin{remark}
		It is perhaps worth explaining why equation (\ref{PBGPTower}) makes sense. Firstly, since $\Sigma \times \Lambda$ is locally compact, we know by \cite[Lemma 46.4]{Munkres} that $\mathcal{L} \rightarrow \mathcal{L}|_{K \times [a,b]}$ is a continuous map from $C(\Sigma \times \Lambda)$ to $C(K \times [a,b])$, so that the left side of (\ref{PBGPTower}) is the conditional expectation of a bounded measurable function, and is thus well-defined. From (\ref{weightedToAvoid}) and \cite[Lemma 3.4]{DM}, the right side is measurable with respect to the $\sigma$-algebra 
		$$ \sigma \left\{ \mathcal{L}_i(s) : \mbox{  $i \in K$ and $s \in \{a,b\}$, or $i \in \{k_1 - 1, k_2 +1 \}$ and $s \in [a,b]$} \right\},$$
		and hence also with respect to $\mathcal{F}_{\mathrm{ext}} (K \times (a,b))$.
	\end{remark} 
	\begin{remark}
		If $A_i = 0$ for all $i \in \Sigma$ then the partial Brownian Gibbs property with respect to $\vec{A}$-area tilts becomes the {\em partial Brownian Gibbs property} from \cite[Definition 2.7]{DM}.
	\end{remark}
	\begin{remark} Definition \ref{DefPBGP} is analogous to \cite[Definition 1]{CIW2} with a few important differences. Firstly, it is defined for line ensembles with an arbitrary number of curves, while \cite[Definition 1]{CIW2} is defined only for infinite line ensembles (although one could modify the definition accordingly). In addition, Definition \ref{DefPBGP} enjoys a nice projectional stability property. Specifically, if $1 \leq M \leq N$ and $\mathcal{L}$ is a $\llbracket 1, N\rrbracket$-indexed line ensemble on $\Lambda$ that satisfies the partial Brownian Gibbs property with respect to $\vec{A}$-area tilts, and $\tilde{\mathcal{L}}$ is obtained from $\mathcal{L}$ by projecting on $(\mathcal{L}_1, \dots, \mathcal{L}_M)$ then the induced law on $\tilde{\mathcal{L}}$ also satisfies the partial Brownian Gibbs property with respect to $\vec{A}_{\llbracket 1, M \rrbracket}$-area tilts as a $\llbracket 1, M \rrbracket$-indexed line ensemble on $\Lambda$. A more important advantage of Definition \ref{DefPBGP} over \cite[Definition 1]{CIW2} is that it applies to line ensembles that can be negative. This is crucial for our setup since the line ensembles $\widetilde{\mathcal{X}}^N$ in Theorem \ref{main} do take on negative values due to the vertical shift. In fact, as we will be shown in Remark \ref{AreaGibbsAiry}, the scaled Airy line ensemble $2^{-1/2}\mathcal{A}^{\mathrm{Airy}}$ from Section \ref{sec1.2} satisfies Definition \ref{DefPBGP} and fails to satisfy \cite[Definition 1]{CIW2} as it is not positive.
	\end{remark}
	
	We end this section with the following lemma, which explains how the partial Brownian Gibbs property with respect to $\vec{A}$-area tilts behaves under affine transformations. 
	
	\begin{lemma}\label{ScalingAGP}Fix a set $\Sigma = \llbracket 1 , N \rrbracket$ with $N \in \mathbb{N}$ or $N  = \infty$ and an interval $\Lambda \subset \mathbb{R}$. In addition, let $A_i \in [0, \infty)$ for $i \in \Sigma$ and suppose that $\mathcal{L}$ is a $\Sigma$-indexed line ensemble on $\Lambda$ that satisfies the partial Brownian Gibbs property with respect to $\vec{A}$-area tilts. Fix $\lambda > 0$, $r, u \in \mathbb{R}$, set $\Lambda_{\lambda, u} = \lambda^{-2} \Lambda + u$, and define the $\Sigma$-indexed line ensemble $\hat{\mathcal{L}}$ on $\Lambda_{\lambda, u}$ by
		\begin{equation*}
			\begin{split}
				&\hat{\mathcal{L}_i} = F_{\lambda,r,u}(\mathcal{L}_i), \mbox{ where } F_{\lambda,r,u}: C(\Lambda) \rightarrow C(\Lambda_{\lambda,u}) \mbox{ is given by } (F_{\lambda,r,u}f)(x) = \lambda^{-1} f(\lambda^{2} (x-u)) + r.
			\end{split}
		\end{equation*} 
		Then, $\hat{\mathcal{L}}$ satisfies the partial Brownian Gibbs property with respect to $\vec{B}$-area tilts, where $\vec{B} = \lambda^3 \vec{A}$.
	\end{lemma}
	\begin{proof} Since $\mathcal{L}$ is non-intersecting by assumption, we conclude the same for $\hat{\mathcal{L}}$. Fix a finite $K = \{k_1, k_1 + 1, \dots, k_2 \} \subset \Sigma$ with $k_2 \leq N - 1$, $[a,b] \subset \Lambda_{\lambda,u}$ and any bounded Borel-measurable function $G: C(K \times [a,b]) \rightarrow \mathbb{R}$. Let $c = \lambda^{2} (a-u)$, $d = \lambda^{2} (b-u)$ and observe that the following two $\sigma$-algebras are identical:
		\begin{align*} \mathcal{F}_{\mathrm{ext}}^{\lambda} &:= \sigma \big \{ \hspace{-0.3mm} \hat{\mathcal{L}}_i(s) \hspace{-0.3mm}:\hspace{-0.3mm} (i,s) \in (\Sigma \times \Lambda_{\lambda,u}) \hspace{-0.3mm}\setminus \hspace{-0.3mm} (K \times (a,b)) \big\},\\ 
			\mathcal{F}_{\mathrm{ext}} &:= \sigma \big \{ \mathcal{L}_i(s) \hspace{-0.3mm} : \hspace{-0.3mm} (i,s) \in (\Sigma \times \Lambda) \hspace{-0.3mm} \setminus \hspace{-0.3mm} (K \times (c,d)) \big\}.
		\end{align*}
		Using that $\mathcal{L}$ satisfies the partial Brownian Gibbs property with respect to $\vec{A}$-area tilts, we conclude 
		\begin{equation}\label{U1}
			\begin{split}
				&\mathbb{E} \left[ G(\hat{\mathcal{L}}|_{K \times [a,b]}) {\big \vert}\mathcal{F}_{\mathrm{ext}}^{\lambda}  \right] = \mathbb{E} \left[ G(\vec{F}_{\lambda,r,u}(\mathcal{L}|_{K \times [c,d]})) {\big \vert}\mathcal{F}_{\mathrm{ext}} \right] = \mathbb{E}_{\mathrm{avoid}, \vec{A}_K}^{k_1, k_2, c,d, \vec{x}, \vec{y}, f, g} \bigl[ G(\vec{F}_{\lambda,r,u}({\mathcal{Q}})) \bigr],
			\end{split}
		\end{equation}
		where $\vec{F}_{\lambda, r, u} : C(K \times [c,d]) \rightarrow C(K \times [a,b])$ is given by
		$$(\vec{F}_{\lambda,r,u}H)(i, x) = \lambda^{-1}  H(i, \lambda^2 (x-u)) + r \mbox{ for } i \in K \mbox{ and } x \in [a,b],$$
		and $\vec{x}$, $\vec{y}$, $f$, $g$, $\vec{A}_K$ are as in the paragraph after (\ref{PBGPTower}). From (\ref{freeToAvoid}) and (\ref{weightedToAvoid}) we have
		\begin{equation}\label{U2}
			\begin{split}
				&\mathbb{E}_{\mathrm{avoid}, \vec{A}_K}^{k_1, k_2, c,d, \vec{x}, \vec{y}, f, g} \bigl[ G(\vec{F}_{\lambda,r,u}({\mathcal{Q}})) \bigr] \hspace{-1mm} = \frac{\mathbb{E}_{\mathrm{avoid}}^{k_1, k_2, c,d, \vec{x}, \vec{y}, \infty, -\infty} \bigl[ G(\vec{F}_{\lambda,r,u}({\mathcal{Q}})) \cdot W_{\vec{A}, g}(\mathcal{Q})\cdot \chi_{f,g} (\mathcal{Q})\bigr] }{\mathbb{E}_{\mathrm{avoid}}^{k_1, k_2, c,d, \vec{x}, \vec{y}, \infty, -\infty} \bigl[ W_{\vec{A}, g}(\mathcal{Q}) \cdot \chi_{f,g} (\mathcal{Q}) \bigr] }, \mbox{ where }\\
				&W_{\vec{A}, g}(\mathcal{Q}) = \exp \left( - \sum_{i = k_1}^{k_2} \int_c^d A_i \cdot  \left( \mathcal{Q}_{i}(s) - g(s)\right)\mathrm{d}s \right) \mbox{ and }  \\
				&\chi_{f,g} (\mathcal{Q}) = {\bf 1}\{ f(x) > \mathcal{Q}_{k_1}(x) \mbox{ and } \mathcal{Q}_{k_2}(x) > g(x) \mbox{ for } x \in [c,d]\}.
			\end{split}
		\end{equation}
		Applying \cite[Lemma 3.5]{DM} to (\ref{U2}), we get
		\begin{equation}\label{U3}
			\begin{split}
				&\mathbb{E}_{\mathrm{avoid}, \vec{A}_K}^{k_1, k_2, c,d, \vec{x}, \vec{y}, f, g} \bigl[ G(\vec{F}_{\lambda,r,u}({\mathcal{Q}})) \bigr] \hspace{-1mm} = \frac{\mathbb{E}_{\mathrm{avoid}}^{k_1, k_2, a,b, \vec{u}, \vec{v}, \infty, -\infty} \bigl[ G(\mathcal{Q}) \cdot \tilde{W}_{\vec{A}, g}(\mathcal{Q})\cdot \chi_{\hat{f},\hat{g}} (\mathcal{Q})\bigr] }{\mathbb{E}_{\mathrm{avoid}}^{k_1, k_2, a,b, \vec{u}, \vec{v}, \infty, -\infty} \bigl[ \tilde{W}_{\vec{A}, g}(\mathcal{Q})\cdot \chi_{\hat{f},\hat{g}} (\mathcal{Q})\bigr] }, \mbox{ where }\\
				&\tilde{W}_{\vec{A}, g}(\mathcal{Q}) = \exp \left( - \sum_{i = k_1}^{k_2} \int_c^d A_i \cdot  \left( \lambda \mathcal{Q}_{i}(\lambda^{-2} s + u) - g(s) - r \lambda \right) \mathrm{d}s \right) \mbox{, }  \\
				& u_i = \lambda (x_i - r) \mbox{, } v_i = \lambda (y_i - r) \mbox{ for $i = k_1, \dots, k_2$,}   \\
				&\hat{f}(x) = \lambda^{-1} f(\lambda^2(x-u)) \mbox{ and } \hat{g}(x) = \lambda^{-1} g(\lambda^2(x-u)) \mbox{ for } x\in [a,b].
			\end{split}
		\end{equation}
		Applying the change of variables $s = \lambda^2(t-u)$ we obtain
		$$\tilde{W}_{\vec{A}, g}(\mathcal{Q})  = D \cdot \exp \left( - \sum_{i = k_1}^{k_2} \int_a^b \lambda^3 A_i \cdot  \left(  \mathcal{Q}_{i}(t) - \hat{g}(t) \right)  \mathrm{d}t \right),$$
		where $D$ is a deterministic factor that depends on $k_1, k_2, r, \lambda, a, b, \vec{A}$. We apply (\ref{freeToAvoid}) and (\ref{weightedToAvoid}) to the right side of (\ref{U3}), factoring out $D$ in the numerator and denominator, to get
		\begin{equation}\label{U4}
			\begin{split}
				&\mathbb{E}_{\mathrm{avoid}, \vec{A}_K}^{k_1, k_2, c,d, \vec{x}, \vec{y}, f, g} \bigl[ G(\vec{F}_{\lambda,r,u}({\mathcal{Q}})) \bigr]  = \mathbb{E}_{\mathrm{avoid}, \vec{B}_K}^{k_1, k_2, a,b, \vec{u}, \vec{v}, \hat{f}, \hat{g}} \bigl[ G(\mathcal{Q}) \bigr],
			\end{split}
		\end{equation}
		where $\vec{B} = \lambda^3\vec{A}$. Combining (\ref{U1}) with (\ref{U4}), we conclude the statement of the lemma. 
	\end{proof}

	%----------------------------------------------------------------------------------------------------------------------------------------------------------------------------------------------------------------------------
	%
	%     Section 2.2
	%
	%----------------------------------------------------------------------------------------------------------------------------------------------------------------------------------------------------------------------------
	\subsection{Area tilts and parabolic shifts}\label{sec2.2} The following statement shows that the ensembles from Definition \ref{DefWeightedLaw} satisfy Definition \ref{DefPBGP}, which in particular produces a large class of ensembles that satisfy the partial Brownian Gibbs property with respect to $\vec{A}$-area tilts.

	\begin{lemma}\label{lemgibbs} Assume the same notation as in Definition \ref{DefWeightedLaw} and suppose that $f = \infty$, $k_1 = 1$, $k_2 = k = N$. Suppose that $(\mathcal{L}_{1}, \dots, \mathcal{L}_{N})$ is distributed according to $\mathbb{P}^{1, N, a,b, \vec{x},\vec{y}, \infty, g}_{\mathrm{avoid}, \vec{A}}$ as in that definition. Then, the $\llbracket 1, N + 1 \rrbracket$-indexed line ensemble $\mathcal{L}$ on $[a,b]$ with curves $\mathcal{L}_1, \dots, \mathcal{L}_N$ and $\mathcal{L}_{N+1} = g$ satisfies the partial Brownian Gibbs property with respect to $\vec{A}$-area tilts from Definition \ref{DefPBGP}.
	\end{lemma}
	\begin{proof} One starts by observing that if $\mathcal{B}_{1}, \dots, \mathcal{B}_{N}$ are distributed according to $\mathbb{P}^{1, N, a,b, \vec{x},\vec{y}, \infty, g}_{\mathrm{avoid}}$ as in Definition \ref{DefAvoidingLaw}, then the $\llbracket 1, N + 1 \rrbracket$-indexed line ensemble $\mathcal{B}$ with the curves $\mathcal{B}_1, \dots, \mathcal{B}_N$ and $\mathcal{B}_{N+1} = g$ satisfies the partial Brownian Gibbs property from \cite[Definition 2.7]{DM}. This can be deduced by a straightforward adaptation of the proof of \cite[Lemma 2.13]{DM} (which corresponds to the case $g = -\infty$). In the sequel we write $\mathbb{P}_{\vec{A}}$ in place of $\mathbb{P}^{1, N, a,b, \vec{x},\vec{y}, \infty, g}_{\mathrm{avoid}, \vec{A}}$ and $\mathbb{P}$ in place of $\mathbb{P}^{1, N, a,b, \vec{x},\vec{y}, \infty, g}_{\mathrm{avoid}}$ to ease notation. We also write $\mathbb{E}_{\vec{A}}$ and $\mathbb{E}$ for the expectations with respect to these measures.\\
		
		Fix $K = \{k_1, k_1 + 1, \dots, k_2 \}$ with $1 \leq k_1 \leq k_2 \leq N$ and $[c,d] \subset [a,b]$. By the defining properties of conditional expectation it suffices to prove that for any $B \in \mathcal{F}_{\mathrm{ext}} (K \times (c,d))$ and bounded Borel-measurable function $F: C(K \times [c,d]) \rightarrow \mathbb{R}$ we have
		\begin{equation}\label{S2P1}
			\mathbb{E}_{\vec{A}} \left[ {\bf 1}_B(\mathcal{L}) \cdot F(\mathcal{L}|_{K \times [c,d]}) \right] = \mathbb{E}_{\vec{A}} \left[ {\bf 1}_B(\mathcal{L})  \cdot \mathbb{E}_{\mathrm{avoid}, \vec{A}_K}^{k_1, k_2, c,d, \vec{x}\hspace{1mm} ', \vec{y} \hspace{1mm} ', f', g'} \bigl[ F({\mathcal{Q}}) \bigr] \right],
		\end{equation}
		where $\vec{x}\hspace{1mm}' = (\mathcal{L}_{k_1}(c), \dots, \mathcal{L}_{k_2}(c))$, $\vec{y}\hspace{1mm} ' = (\mathcal{L}_{k_1}(d), \dots, \mathcal{L}_{k_2}(d))$, $f' = \mathcal{L}_{k_1 - 1}[c,d]$ (the restriction of $\mathcal{L}$ to the set $\{k_1 - 1 \} \times [c,d]$) with the convention that $f = \infty$ if $k_1  = 1$, and $g = \mathcal{L}_{k_2 +1}[c,d]$. By a standard monotone class argument, see e.g. Step 1 of the proof of \cite[Lemma 2.13]{DM}, it suffices to prove (\ref{S2P1}) only for events $B$ of the form
		$$B = \{ h \in C(\llbracket 1, N\rrbracket \times [a,b]): h(n_i, t_i) \leq a_i,\mbox{ }(n_i,t_i) \in D_{K,c,d}^c  \mbox{ for } i\in \llbracket 1, m\rrbracket \}.$$    
		
		From (\ref{weightedToAvoid}) we know that 
		\begin{equation}\label{S2P2}
			\begin{split}
				&\mathbb{E}_{\vec{A}} \left[ {\bf 1}_{B}(\mathcal{L}) \cdot F(\mathcal{L}|_{K \times [c,d]}) \right] = \frac{\mathbb{E} \left[ {\bf 1}_{B}(\mathcal{B}) \cdot  W_{\vec{A}} (\mathcal{B}) \cdot F(\mathcal{B}|_{K \times [c,d]})  \right]}{\mathbb{E}\left[W_{\vec{A}} (\mathcal{B})\right]} \mbox{, and }\\
				&\mathbb{E}_{\vec{A}} \left[  {\bf 1}_{B}(\mathcal{L})   \cdot \mathbb{E}_{\mathrm{avoid}, \vec{A}_K}^{k_1, k_2, c,d, \vec{x}\hspace{1mm} ', \vec{y} \hspace{1mm} ', f', g'} \bigl[ F({\mathcal{Q}}) \bigr] \right] = \frac{\mathbb{E} \left[  {\bf 1}_{B}(\mathcal{B}) \cdot  W_{\vec{A}} (\mathcal{B}) \cdot \mathbb{E}_{\mathrm{avoid}, \vec{A}_K}^{k_1, k_2, c,d, \vec{u}, \vec{v}, h_+, h_- } \bigl[ F({\mathcal{Q}}) \bigr]  \right]}{\mathbb{E}\left[W_{\vec{A}} (\mathcal{B})\right]},
			\end{split}
		\end{equation}
		where $\vec{u} = (\mathcal{B}_{k_1}(c), \dots, \mathcal{B}_{k_2}(c))$, $\vec{v} = (\mathcal{B}_{k_1}(d), \dots, \mathcal{B}_{k_2}(d))$, $h_+ =   \mathcal{B}_{k_1-1}[c,d] $ (or $h_+ = \infty$ if $k_1 =1 $), $h_- = \mathcal{B}_{k_2+1}[c,d]$, and
		\begin{equation}\label{S2P3}
			\begin{split}
				&W_{\vec{A}} (\mathcal{B}) = \exp \left( - \sum_{i = 1}^{N} \int_a^b A_i \cdot \left( \mathcal{B}_{i}(u) - g(u)\right) \mathrm{d}u \right) = W^1_{\vec{A}} (\mathcal{B}) \cdot W^2_{\vec{A}} (\mathcal{B}), \mbox{ where }\\
				& W^1_{\vec{A}} (\mathcal{B}) = \exp \left( - \sum_{i \in K} \int_c^d A_i \cdot \left( \mathcal{B}_{i}(u) - g(u)\right) \mathrm{d}u \right) \mbox{ and }\\
				& W^2_{\vec{A}} (\mathcal{B}) = \exp \left( - \sum_{i \in \llbracket 1, N \rrbracket \setminus K} \int_a^b A_i \cdot \left( \mathcal{B}_{i}(u) - g(u)\right) \mathrm{d}u - \sum_{i \in K} \int_{[a,c] \cup [d,b]} \hspace{-5mm} A_i \cdot \left( \mathcal{B}_{i}(u) - g(u)\right) \mathrm{d}u \right).
			\end{split}
		\end{equation}
		
		Notice that $W^2_{\vec{A}} (\mathcal{B}) $ is $\mathcal{F}_{\mathrm{ext}} (K \times (c,d))$-measurable and so by the partial Brownian Gibbs property (satisfied by $\mathcal{B}$) we have
		\begin{equation}\label{S2P4}
			\begin{split}
				&\mathbb{E} \left[ {\bf 1}_{B}(\mathcal{B}) \cdot  W_{\vec{A}} (\mathcal{B}) \cdot F(\mathcal{B}|_{K \times [c,d]})  \right] = \\
				&\mathbb{E} \left[ {\bf 1}_{B}(\mathcal{B}) \cdot  W^2_{\vec{A}} (\mathcal{B}) \cdot \mathbb{E}_{\mathrm{avoid}}^{k_1, k_2, c,d, \vec{u}, \vec{v}, h_+, h_-} \bigl[ W^1_{\vec{A}} (\mathcal{Q})\cdot F({\mathcal{Q}}) \bigr] \right].
			\end{split}
		\end{equation}
		We next notice that 
		\begin{equation*}
			\begin{split}
				&W^1_{\vec{A}} (\mathcal{Q})  = \tilde{W}^1_{\vec{A}} (\mathcal{Q}) \cdot G(\mathcal{B}_{k_2+1}), \mbox{ where } G(\mathcal{B}_{k_2+1}) = \exp \left( - \sum_{i \in K} \int_c^d A_i \cdot \left( \mathcal{B}_{k_2+1}(u) - g(u)\right)\mathrm{d}u \right) , \mbox{ and }\\
				&\tilde{W}^1_{\vec{A}} (\mathcal{Q}) = \exp \left( - \sum_{i \in K} \int_c^d A_i \cdot \left( \mathcal{Q}_{i}(u) - \mathcal{B}_{k_2+1}(u)\right)\mathrm{d}u \right).
			\end{split}
		\end{equation*}
		Combining the latter with (\ref{weightedToAvoid}) we conclude that 
		\begin{equation*}
			\begin{split}
				&\mathbb{E}_{\mathrm{avoid}}^{k_1, k_2, c,d, \vec{u}, \vec{v}, h_+, h_-} \bigl[ W^1_{\vec{A}} (\mathcal{Q})\cdot F({\mathcal{Q}}) \bigr]  =\mathbb{E}_{\mathrm{avoid}, \vec{A}_K}^{k_1, k_2, c,d, \vec{u}, \vec{v}, h_+, h_-} \bigl[  F({\mathcal{Q}}) \bigr]  \cdot   G(\mathcal{B}_{k_2+1}) \cdot \mathbb{E}_{\mathrm{avoid}}^{k_1, k_2, c,d, \vec{u}, \vec{v}, h_+, h_-} \bigl[  \tilde{W}^1_{\vec{A}} (\mathcal{Q}) \bigr].
			\end{split}
		\end{equation*}
		The latter and (\ref{S2P4}) give
		\begin{equation}\label{S2P5}
			\begin{split}
				&\mathbb{E} \left[ {\bf 1}_{B}(\mathcal{B}) \cdot  W_{\vec{A}} (\mathcal{B}) \cdot F(\mathcal{B}|_{K \times [c,d]})  \right] = \\
				&\mathbb{E} \left[ {\bf 1}_{B}(\mathcal{B}) \cdot  W^2_{\vec{A}} (\mathcal{B}) \cdot  G(\mathcal{B}_{k_2+1}) \cdot \mathbb{E}_{\mathrm{avoid}}^{k_1, k_2, c,d, \vec{u}, \vec{v}, h_+, h_-} \bigl[  \tilde{W}^1_{\vec{A}} (\mathcal{Q}) \bigr] \cdot \mathbb{E}_{\mathrm{avoid}, \vec{A}_K}^{k_1, k_2, c,d, \vec{u}, \vec{v}, h_+, h_-} \bigl[  F({\mathcal{Q}}) \bigr]\right].
			\end{split}
		\end{equation}
		
		Arguing in an analogous way for the second line in (\ref{S2P2}) (i.e. replacing $F(\mathcal{B}|_{K \times [c,d]})$ with $\mathbb{E}_{\mathrm{avoid}, \vec{A}_K}^{k_1, k_2, c,d, \vec{u}, \vec{v}, h_+, h_- } \bigl[ F({\mathcal{Q}}) \bigr] $) we get that the numerator on the right side of the second line in (\ref{S2P2}) is equal to the second line in (\ref{S2P5}). This proves that the two lines in (\ref{S2P2}) are equal, which proves (\ref{S2P1}) and hence the lemma.
	\end{proof}
	
	In the remainder of this section we only deal with the case when $\vec{A}$ is such that $A_i = A$ for all $i \in \mathbb{N}$ and $A \in [0, \infty)$. We seek to show that if an ensemble satisfies the partial Brownian Gibbs property with respect to $\vec{A}$-area tilts, then upon subtracting the same parabola with curvature $A$ from each curve the resulting ensemble satisfies the partial Brownian Gibbs property from \cite[Definition 2.7]{DM} (i.e. the property in Definition \ref{DefPBGP} with all $A_i = 0$). This statement is a consequence of the Cameron-Martin-Girsanov theorem (see \cite{CamMar44} or \cite[Chapter 3, Theorem 5.1]{KS} for a textbook treatment), but we write out the full argument starting with the following core lemma.
	
	\begin{lemma}\label{finParLemma} Let $k_1, k_2, k, a, b, \vec{x}, \vec{y},f,g, \Sigma$ be as in Definition \ref{DefWeightedLaw} and fix $A \in [0, \infty)$, $c,d \in \mathbb{R}$. Let $\mathcal{L}$ be a $\Sigma$-indexed line ensemble on $[a,b]$ with distribution $\mathbb{P}_{\mathrm{avoid}}^{k_1, k_2, a,b, \vec{x}, \vec{y}, f, g}$ as in Definition \ref{DefAvoidingLaw}. Let $h(s) = (A/2)s^2 + cs + d$ and define the $\Sigma$-indexed line ensemble $\tilde{\mathcal{L}}$ on $[a,b]$ via 
		$$\tilde{\mathcal{L}}_i(s) = \mathcal{L}_i(s) + h(s) \mbox{ for $s \in [a,b]$ and $i \in \Sigma$}.$$
		Define $\vec{u}, \vec{v}$ via $u_i = x_i + h(a)$, $v_i = y_i + h(b)$ for $i = 1, \dots, k$, and set $\tilde{f}(s) = f(s) + h(s)$, $\tilde{g}(s) = g(s) + h(s)$ $($with the convention that $\tilde{f} = \infty$ if $f = \infty)$. Then, the law of $\tilde{\mathcal{L}}$ is $\mathbb{P}_{\mathrm{avoid}, \vec{A}}^{k_1, k_2, a,b, \vec{u}, \vec{v}, \tilde{f}, \tilde{g}}$ as in Definition \ref{DefWeightedLaw} with $A_i = A$ for $i \in \Sigma$.
	\end{lemma}
	\begin{proof} Let $W_t$ be a standard Brownian motion on $[0,1]$, defined on a probability space $(\Omega, \mathcal{F}, \mathbb{P})$. From the Cameron-Martin-Girsanov theorem we have that if $\phi \in C^2([0,1])$, and $\phi(0) = 0$, then 
		$$\tilde{W}_t := W_t - \phi(t)$$
		is a standard Brownian motion on the space $(\Omega, \mathcal{F}, \mathbb{Q})$, where 
		$$\frac{\mathrm{d} \mathbb{Q}}{\mathrm{d} \mathbb{P}} = \exp \left( \int_0^1 \phi'(s) \,\mathrm{d}W_s - \int_0^1 [\phi'(s)]^2 \,\mathrm{d}s \right).$$
		Since $\phi \in C^2([0,1])$ is deterministic, we see using integration by parts that 
		\begin{equation}\label{GirsanovE1}
			\begin{split}
				&\frac{\mathrm{d} \mathbb{Q}}{\mathrm{d} \mathbb{P}} = \frac{1}{Z_\phi}\exp \left( \phi'(1)  W_1 - \int_0^1 \phi''(s)W_s \,\mathrm{d}s  \right), \mbox{ where }\\
				& Z_{\phi} = \mathbb{E}_{\mathbb{P}} \left[ \exp \left( \phi'(1)  W_1 - \int_0^1 \phi''(s)W_s \,\mathrm{d}s  \right)\right].
			\end{split}
		\end{equation}
		
		For $a,b$ as in the statement of the lemma, and $x, y \in \mathbb{R}$ we define $F_{x,y}: C([0,1]) \rightarrow C([a,b])$ via 
		$$F_{x,y}(\psi)(t) = (b-a)^{1/2} \left[\psi \left( \frac{t - a}{b-a} \right) - \frac{t- a}{b-a} \cdot \psi(1) \right] + \frac{b-t}{b-a} \cdot x +  \frac{t- a}{b-a} \cdot y.$$
		We also fix $\phi(s) = (b-a)^{3/2} (A/2) \cdot s^2$ and observe that 
		\begin{equation}\label{GirsanovE2}
			F_{x+h(a), y + h(b)} (W)(t) = F_{x,y}(\tilde{W})(t)  + h(t) \mbox{ for $t \in [a,b]$}.
		\end{equation}
		One way to quickly check the latter is to note that the difference of the two sides is a linear (in $t$) function, and then check that the two sides agree when $t = a$ and $t = b$. 
		
		Using (\ref{GirsanovE1}), we have for any bounded measurable $G: C([a,b]) \rightarrow \mathbb{R}$ that
		\begin{equation}\label{GirsanovE3}
			\mathbb{E}_{\mathbb{Q}} \left[ G(F_{x,y}(\tilde{W}) + h) \right] = \frac{\mathbb{E}_{\mathbb{P}} \left[ G(F_{x,y}(\tilde{W}) + h) \cdot \exp \left( \phi'(1)  W_1 - \int_0^1 \phi''(s)W_s\, \mathrm{d}s \right) \right]}{\mathbb{E}_{\mathbb{P}} \left[ \exp \left( \phi'(1)  W_1 - \int_0^1 \phi''(s)W_s \,\mathrm{d}s \right)\right]}.
		\end{equation}
		
		In view of (\ref{BBDef}), we have that under $\mathbb{Q}$ the curve $F_{x,y}(\tilde{W})$ is a Brownian bridge (from $B(a) = x$ to $B(b) = y$) with diffusion parameter $1$. Furthermore,  from (\ref{GirsanovE2}) we have that under $\mathbb{P}$ the curve $F_{x,y}(\tilde{W}) + h$ is a Brownian bridge (from $B(a) = x + h(a)$ to $B(b) = y + h(b)$) with diffusion parameter $1$. Finally, a direct computation using that $\phi(s) = (b-a)^{3/2}(A/2) \cdot s^2$ gives 
		$$\phi'(1)  W_1 - \int_0^1 \phi''(s)W_s \, \mathrm{d}s = - A \int_a^b F_{x+h(a), y + h(b)} (W)(t) \, \mathrm{d}t + \kappa,$$
		where $\kappa$ is a deterministic constant (it depends on $a,b,x,y,h$). Combining all of these observations and (\ref{GirsanovE3}) we conclude that
		\begin{equation}\label{GirsanovE4}
			\mathbb{E}_{\mathrm{free}}^{a,b,x,y} \left[ G(\mathcal{Q}_1 + h) \right] = \frac{\mathbb{E}^{a,b,x + h(a),y + h(b)}_{\mathrm{free}} \left[ G(\tilde{\mathcal{Q}}_1)  \exp \left( - A \int_a^b \tilde{\mathcal{Q}}_1(t)\, \mathrm{d}t  \right) \right]}{\mathbb{E}^{a,b,x + h(a),y + h(b)}_{\mathrm{free}} \left[ \exp \left( - A \int_a^b \tilde{\mathcal{Q}}_1(t)\, \mathrm{d}t  \right) \right]},
		\end{equation}
		where $\mathcal{Q}_1$ has law $\mathbb{P}_{\mathrm{free}}^{a,b,x,y}$ and $\tilde{\mathcal{Q}}_1$ has law $\mathbb{P}^{a,b,x + h(a),y + h(b)}_{\mathrm{free}}$.
		
		Using (\ref{GirsanovE4}), independence and a monotone class argument, we see that for any bounded measurable $G: C(\Sigma \times [a,b]) \rightarrow \mathbb{R}$,
		\begin{equation}\label{GirsanovE5}
			\mathbb{E}_{\mathrm{free}}^{k_1, k_2, a,b,\vec{x},\vec{y}} \left[ G(\mathcal{Q} + h) \right] = \frac{\mathbb{E}^{k_1, k_2, a,b,\vec{u}, \vec{v}}_{\mathrm{free}} \left[ G(\tilde{\mathcal{Q}})  \exp \left( - A \sum_{i = k_1}^{k_2}  \int_a^b \tilde{\mathcal{Q}}_i(t)\, \mathrm{d}t  \right) \right]}{\mathbb{E}^{k_1, k_2, a,b,\vec{u},\vec{v}}_{\mathrm{free}} \left[ \exp \left( - A \sum_{i = k_1}^{k_2} \int_a^b \tilde{\mathcal{Q}}_i(t)\, \mathrm{d}t  \right) \right]},
		\end{equation}
		where $\mathcal{Q}$ is $\mathbb{P}_{\mathrm{free}}^{k_1, k_2, a,b,\vec{x},\vec{y}}$-distributed, $\tilde{\mathcal{Q}}$ is $\mathbb{P}^{k_1, k_2, a,b,\vec{u},\vec{v}}_{\mathrm{free}}$-distributed and $\mathcal{Q} + h$ stands for the ensemble obtained by $\mathcal{Q}$ by adding to each curve $h$. 
		
		We now take the ratio of (\ref{GirsanovE5}) when $G = H \cdot \chi$ and $G = \chi$, where $H: C(\Sigma \times [a,b]) \rightarrow \mathbb{R}$ is bounded and measurable and for $\Psi \in C(\Sigma \times [a,b])$ we have
		$$\chi(\Psi) = {\bf1} \left\{ f(r) > \Psi(k_1, r) >  \Psi(k_1 + 1, r) > \cdots >  \Psi(k_2, r) > g(r) \mbox{ for all $r \in[a,b]$} \right\}.$$ 
		In view of (\ref{freeToAvoid}) the result is 
		\begin{equation*}
			\mathbb{E}_{\mathrm{avoid}}^{k_1, k_2, a,b,\vec{x},\vec{y},f,g} \left[ H(\mathcal{Q} + h) \right] = \frac{\mathbb{E}^{k_1, k_2, a,b,\vec{u}, \vec{v},f,g}_{\mathrm{avoid}} \left[ H(\tilde{\mathcal{Q}})  \exp \left( - A \sum_{i = k_1}^{k_2}  \int_a^b \tilde{\mathcal{Q}}_i(t)\, \mathrm{d}t  \right) \right]}{\mathbb{E}^{k_1, k_2, a,b,\vec{u}, \vec{v},f,g}_{\mathrm{avoid}} \left[ \exp \left( - A \sum_{i = k_1}^{k_2}  \int_a^b \tilde{\mathcal{Q}}_i(t)\, \mathrm{d}t  \right) \right]}.
		\end{equation*}
		The latter equation, (\ref{weightedToAvoid}), and the fact that $g$ is deterministic give the statement of the lemma.
	\end{proof}
	
	We now turn to the main result of this section.
	
	\begin{lemma}\label{GibbsPar} Fix a set $\Sigma = \llbracket 1 , N \rrbracket$ with $N \in \mathbb{N}$ or $N  = \infty$ and an interval $\Lambda \subset \mathbb{R}$. In addition, let $A \in [0, \infty)$, $c, d\in \mathbb{R}$ and put $h(x) = (A/2)x^2 + cx + d$ for $x \in \mathbb{R}$. Suppose that $\mathcal{L}$ is a $\Sigma$-indexed line ensemble on $\Lambda$ and define the $\Sigma$-indexed line ensemble $\tilde{\mathcal{L}}$ on $\Lambda$ by  
		$$\tilde{\mathcal{L}}_i(s) = \mathcal{L}_i(s) + h(s) \mbox{ for $s \in \Lambda$ and $i \in \Sigma$}.$$
		Then, $\mathcal{L}$ satisfies the partial Brownian Gibbs property if and only if $\tilde{\mathcal{L}}$ satisfies the partial Brownian Gibbs property with respect to $\vec{A}$-area tilts, where $A_i = A$ for $i \in \Sigma$.
	\end{lemma}
	\begin{remark}\label{AreaGibbsAiry}
		From \cite[Theorem 3.1]{CH14} we have that the parabolic Airy line ensemble $\mathcal{L}^{\mathrm{Airy}}$ from Section \ref{sec1.2} satisfies the partial Brownian Gibbs property (i.e. Definition \ref{DefPBGP} with $A_i = 0$ for $i \in \mathbb{N}$). From Lemma \ref{GibbsPar} we conclude that $\mathcal{L}^{\mathrm{Airy}}_i(x) + 2^{-1/2}x^2$, which equals $2^{-1/2}\mathcal{A}^{\mathrm{Airy}}_i(x), $ satisfies the partial Brownian Gibbs property with respect to $\vec{A}$-area tilts, where $A_i = 2^{1/2}$ for $i \in \mathbb{N}$.
	\end{remark}
	\begin{proof} As the two directions are analogous, we only prove the forward direction, i.e., we assume that $\mathcal{L}$ satisfies the partial Brownian Gibbs property and proceed to show that $\tilde{\mathcal{L}}$ satisfies the partial Brownian Gibbs property with respect to $\vec{A}$-area tilts. Since $\mathcal{L}$ is non-intersecting (this is part of the statement that it satisfies the partial Brownian Gibbs property), we conclude the same for $\tilde{\mathcal{L}}$. 
		
		Let us fix a finite $K = \{k_1, k_1 + 1, \dots, k_2 \} \subset \Sigma$ with $k_2 \leq N - 1$ (if $\Sigma \neq \mathbb{N}$), $[a,b] \subset \Lambda$ and any bounded Borel-measurable function $F: C(K \times [a,b]) \rightarrow \mathbb{R}$. We let $T_h: C(K \times [a,b]) \rightarrow C(K \times [a,b]) $ be the (continuous) function such that for $\Psi \in C(K \times [a,b])$,
		$$T_h(\Psi)(i,s) = \Psi(i,s) + h(s) \mbox{ for $i \in K$ and $s \in [a,b]$}.$$
		From the partial Brownian Gibbs property, we know that $\mathbb{P}$-almost surely
		\begin{equation}\label{W1}
			\mathbb{E} \left[ F(T_h(\mathcal{L}|_{K \times [a,b]})) {\big \vert} \mathcal{F}_{\mathrm{ext}} (K \times (a,b))  \right] =\mathbb{E}_{\mathrm{avoid}}^{k_1, k_2, a,b, \vec{x}, \vec{y}, f, g} \bigl[ F(T_h(\mathcal{Q})) \bigr],
		\end{equation}
		where $D_{K,a,b} = K \times (a,b)$ and $D_{K,a,b}^c = (\Sigma \times \Lambda) \setminus D_{K,a,b}$,
		$$\mathcal{F}_{\mathrm{ext}} (K \times (a,b)) = \sigma \left \{ \mathcal{L}_i(s): (i,s) \in D_{K,a,b}^c \right\},$$
		$\vec{x} = (\mathcal{L}_{k_1}(a), \dots, \mathcal{L}_{k_2}(a))$, $\vec{y} = (\mathcal{L}_{k_1}(b), \dots, \mathcal{L}_{k_2}(b))$, $f = \mathcal{L}_{k_1 - 1}[a,b]$ with the convention that $f = \infty$ if $k_1 - 1 \not \in \Sigma$, and $g = \mathcal{L}_{k_2 +1}[a,b]$. Using Lemma \ref{finParLemma} and the definition of $\tilde{\mathcal{L}}$ we see that (\ref{W1}) is equivalent to
		\begin{equation}\label{W2}
			\mathbb{E} \left[ F(\tilde{\mathcal{L}}|_{K \times [a,b]}) {\big \vert} \mathcal{F}_{\mathrm{ext}} (K \times (a,b))  \right] =\mathbb{E}_{\mathrm{avoid}, \vec{A}}^{k_1, k_2, a,b, \vec{u}, \vec{v},\tilde{f}, \tilde{g}} \bigl[ F(\mathcal{Q}) \bigr],
		\end{equation}
		where $\vec{u} = (\tilde{\mathcal{L}}_{k_1}(a), \dots, \tilde{\mathcal{L}}_{k_2}(a))$, $\vec{v} = (\tilde{\mathcal{L}}_{k_1}(b), \dots, \tilde{\mathcal{L}}_{k_2}(b))$, $\tilde{f} = \tilde{\mathcal{L}}_{k_1 - 1}[a,b]$ with the convention that $\tilde{f} = \infty$ if $k_1 - 1 \not \in \Sigma$, and $\tilde{g} = \tilde{\mathcal{L}}_{k_2 +1}[a,b]$. Finally, we observe that since $\tilde{\mathcal{L}}$ is a deterministic shift of $\mathcal{L}$, we have 
		$$\mathcal{F}_{\mathrm{ext}} (K \times (a,b)) = \sigma \left \{ \tilde{\mathcal{L}}_i(s): (i,s) \in D_{K,a,b}^c \right\},$$
		which together with (\ref{W2}) shows that $\mathcal{L}$ satisfies (\ref{PBGPTower}). As $K,a,b$ were arbitrary, this proves the lemma.
	\end{proof}

	%----------------------------------------------------------------------------------------------------------------------------------------------------------------------------------------------------------------------------
	%
	%     Section 3
	%
	%-------------------------------------------------------------------------------------------------------------------------------------------------------------------------------------------------------------------
	\section{Gibbs property for the Dyson Ferrari--Spohn diffusion}\label{sec3} In this section we seek to prove that the Dyson Ferrari--Spohn diffusion $\mathcal{X}^N$ from Section \ref{sec1.2} satisfies the partial Brownian Gibbs property with respect to area tilts from Definition \ref{DefPBGP}. The precise statement can be found in Proposition \ref{DFSGibbs}, and its proof relies on \cite[Theorem A]{IVW}, which shows that a certain rescaled version of $\mathcal{X}^N$ can be obtained as a weak limit of discrete line ensembles of avoiding random walks with discrete area tilts. The essential difficulty lies in showing that the discrete area tilt Gibbs property becomes the partial Brownian Gibbs property with respect to $\vec{A}$-area tilts in the limit that takes the random walks to Brownian motions. 
	
	%----------------------------------------------------------------------------------------------------------------------------------------------------------------------------------------------------------------------------
	%
	%     Section 3.1
	%
	%-------------------------------------------------------------------------------------------------------------------------------------------------------------------------------------------------------------------
	\subsection{Discrete area tilts}\label{sec3.1} In this section we define certain line ensembles of discrete random walks with area tilts, which will be used to approximate a rescaled version of $\mathcal{X}^N$. We also show that if the boundary data of these ensembles converges under an appropriate scaling, then one recovers the measures in Definition \ref{DefWeightedLaw}. We mention here that our choice for discrete ensembles is quite limited (i.e. we will pick a very specific jump distribution and area tilts), which is made so that the results of \cite{IVW} and \cite{Ser1} are directly applicable. It is possible to generalize the results of this section to quite generic random walks, but we do not pursue this here.
	
	\begin{definition}\label{walktilt}
		Fix $k_1,k_2\in\mathbb{N}$ with $k_1\leq k_2$ and set $k = k_2-k_1+1$. Fix $a,b\in\mathbb{Z}$ with $a<b$, $\vec{x},\vec{y}\in W_k^\circ \cap \mathbb{N}^k$ with $|x_i-y_i|\leq 2(b-a)$ for each $i \in \llbracket k_1, k_2 \rrbracket$. We also fix $g : \llbracket a,b\rrbracket \to \mathbb{Z}_{\geq 0}$ and $f : \llbracket a, b \rrbracket \rightarrow \mathbb{Z}_{\geq 0} \cup \{ \infty\}$ such that $f(x) > g(x)$ for $x \in \llbracket a,b\rrbracket$. Let $X = (X_{k_1},\dots,X_{k_2})$ denote $k$ independent random walk bridges on $[a,b]$ with $X(a)=\vec{x}$ and $X(b)=\vec{y}$, with increments distributed according to a fixed probability measure with support $\{0,\pm 1, \pm 2\}$ which has mean $0$ and variance $1$. Let $\mathbb{P}^{k_1,k_2,a,b,\vec{x},\vec{y},f,g}_{\mathrm{avoid},\mathrm{walk}}$ denote the law of $X$, conditioned on the event $\{f(j) > X_{k_1}(j) > \cdots > X_{k_2}(j) > g(j) \mbox{ for all } j\in\llbracket a,b\rrbracket\}$, when this event is non-empty and hence of positive probability. Fix $\lambda \geq 0$, and define the probability measures $\mathbb{P}^{k_1,k_2,a,b,\vec{x},\vec{y},f,g}_{\mathrm{avoid},\mathrm{walk},\lambda}$ by
		\begin{equation}\label{walkareadef}
			\frac{\mathrm{d} \mathbb{P}^{k_1,k_2,a,b,\vec{x},\vec{y},f,g}_{\mathrm{avoid},\mathrm{walk}, \lambda}}{\mathrm{d}\mathbb{P}^{k_1,k_2,a,b,\vec{x},\vec{y},f,g}_{\mathrm{avoid}, \mathrm{walk}}} = \frac{1}{Z^{k_1,k_2,a,b,\vec{x},\vec{y},f,g}_{\mathrm{avoid}, \mathrm{walk},\lambda}} \exp\left(-\sum_{i=k_1}^{k_2} \lambda \sum_{j=a}^{b-1} X_i(j)\right).
		\end{equation}
		Note we need not subtract the bottom curve $g$ in the area tilt here, as we have assumed $g\geq 0$ and thus conditioned on $X_i$ being strictly positive for all $i$. This implies that $Z^{k_1,k_2,a,b,\vec{x},\vec{y},f,g}_{\mathrm{avoid,walk},\lambda} \in (0,1]$ as long as $Z^{k_1,k_2,a,b,\vec{x},\vec{y},f,g}_{\mathrm{avoid,walk}} > 0$, and in this case the measures are well-defined.
	\end{definition}
	
	The measures in Definition \ref{walktilt} are {\em a priori} measures on a finite-dimensional vector space, indexed by $\llbracket k_1, k_2 \rrbracket \times \llbracket a, b \rrbracket$. By linearly interpolating the points $(x, X_i(x))$ for $x \in \llbracket a ,b \rrbracket$ and each $i \in \llbracket k_1, k_2 \rrbracket$ we can naturally view the measures from Definition \ref{walktilt} as ones on the space $C(\llbracket k_1, k_2 \rrbracket \times [a,b])$. That is, $X$ can be viewed as a $\llbracket k_1, k_2 \rrbracket$-indexed line ensemble on $[a,b]$ in the sense of Definition \ref{DefLE}.

	\begin{remark}\label{discgibbs}
		It is straightforward to see that the measures in Definition \ref{walktilt} satisfy a Gibbs property analogous to that in Definition \ref{DefPBGP}, which we will call the \textit{Gibbs property with respect to discrete area tilts}. Namely, let $X$ have law $\mathbb{P}^{k_1,k_2,a,b,\vec{x},\vec{y},f,g}_{\mathrm{avoid,walk},\lambda}$. If $K = \llbracket k_1',k_2'\rrbracket \subseteq \llbracket k_1,k_2\rrbracket$, $\Lambda' = \llbracket c,d\rrbracket \subseteq \llbracket a,b\rrbracket$ with $c,d \in \mathbb{Z}, c < d$, and $\mathcal{F}_{\mathrm{ext}}^X(K\times\Lambda')$ is the $\sigma$-algebra generated by $X_i(j)$ for $i\in K'$ and $j\notin \llbracket c+1,d-1\rrbracket$ and by $X_i(j)$ for $i\notin K'$ and $j \in \llbracket a, b \rrbracket$, then
		\begin{equation}
			\mathrm{Law}\left(X|_{K\times \Lambda'} \, \big| \, \mathcal{F}_{\mathrm{ext}}^X(K'\times\Lambda')\right) = \mathbb{P}_{\mathrm{avoid},\mathrm{walk},\lambda}^{k_1',k_2',c,d,\vec{x}\,',\vec{y}\,',f',g'},
		\end{equation}
		where $\vec{x}\,' = (X_{k_1'}(c), \dots, X_{k_2'}(c))$, $\vec{y}\,' = (X_{k_1'}(d), \cdots, X_{k_2'}(d))$, $f' = X_{k_1'-1}\llbracket c,d\rrbracket$ (with $X_0 = f$), $g' = X_{k_2'+1}\llbracket c,d\rrbracket$ (with $X_{k_2 + 1} = g$).
	\end{remark}

	We end this section with the following technical lemma, which shows that under appropriate scaling the measures in Definition \ref{walktilt} converge to those of Definition \ref{DefWeightedLaw}. The lemma is quite similar to \cite[Lemma 2.1]{Ser2} and \cite[Lemma 5.5]{DM}, but we state and prove it here in the present notation for the sake of completeness.
	\begin{lemma}\label{walkconv}
		Fix $k_1,k_2,k,a,b,\vec{x},\vec{y},f,g$ as in Definition \ref{DefWeightedLaw} and assume $g \geq 0$. Let $a_M = \lfloor aM^{2/3}\rfloor$, $b_M = \lceil bM^{2/3}\rceil$. Suppose $\vec{x}_M,\vec{y}_M$ are sequences in $W_k^\circ \cap \mathbb{N}^k$ such that $M^{-1/3}\vec{x}_M \to \vec{x}$ and $M^{-1/3}\vec{y}_M \to \vec{y}$ as $M\to\infty$. Let $g_M: \llbracket a_M, b_M\rrbracket \to \mathbb{Z}_{\geq 0}$ and $f_M:\llbracket a_M, b_M\rrbracket \to \mathbb{Z}_{\geq 0} \cup \{\infty\} $ be such that and $M^{-1/3}g_M(tM^{2/3})\to g(t)$ and $M^{-1/3}f_M(tM^{2/3})\to f(t)$ uniformly in $t\in[a,b]$ (if $f \equiv \infty$, this means that $f_M = \infty$ for all large $M$). Then the following hold:
		\begin{enumerate}
			\item There exists $M_0 \in \mathbb{N}$ such that for $M \geq M_0$ the measures $\mathbb{P}^{k_1,k_2,a_M,b_M,\vec{x}_M,\vec{y}_M,f_M,g_M}_{\mathrm{avoid,walk},1/M}$ as in Definition \ref{walktilt} are well-defined.
			\item If $Z^M$ has law $\mathbb{P}^{k_1,k_2,a_M,b_M,\vec{x}_M,\vec{y}_M,f_M,g_M}_{\mathrm{avoid,walk},1/M}$ for $M \geq M_0$ and we define
			\[
			\mathcal{Z}^M(t) = M^{-1/3}Z^M(tM^{2/3}), \quad t\in[a,b],
			\]
			as a $\llbracket k_1, k_2 \rrbracket$-indexed line ensemble on $[a,b]$, then $\mathcal{Z}^M$ converges in law to $\mathbb{P}^{k_1,k_2,a,b,\vec{x},\vec{y},f,g}_{\mathrm{avoid},\vec{A}}$ as $M\to\infty$, where $A_i = 1$ for $i\in \llbracket k_1,k_2\rrbracket$.
		\end{enumerate}
	\end{lemma}
	
	\begin{proof}
		
		From \cite[Lemma 3.10]{Ser1} (with $\llbracket\alpha,\beta\rrbracket = \llbracket -2,2\rrbracket$, $p=0$, $\sigma = 1$, and $T = M^{2/3}$), we have that there exists $M_0 \in \mathbb{N}$ such that for $M \geq M_0$ the event $\{f_M(j) > X_{k_1}(j) > \cdots > X_{k_2}(j) > g_M(j) \mbox{ for all } j\in\llbracket a,b\rrbracket\}$ is non-empty, which means that both $\mathbb{P}^{k_1,k_2,a_M,b_M,\vec{x}_M,\vec{y}_M,f_M,g_M}_{\mathrm{avoid,walk}}$ and $\mathbb{P}^{k_1,k_2,a_M,b_M,\vec{x}_M,\vec{y}_M,f_M,g_M}_{\mathrm{avoid,walk},1/M}$ are well-defined for such $M$, and we assume $M \geq M_0$ in the sequel.
		
		Let $Q^M$ have law $\mathbb{P}^{k_1,k_2,a_M,b_M,\vec{x}_M,\vec{y}_M,f_M,g_M}_{\mathrm{avoid,walk}}$, and define $\mathcal{Q}^M(t) = M^{-1/3}Q^M(tM^{2/3})$ for $tM^{2/3}\in[a_M,b_M]$. From \cite[Lemma 3.10]{Ser1}, $\mathcal{Q}^M|_{\llbracket k_1,k_2\rrbracket\times[a,b]}$ converges weakly to a random variable $\mathcal{Q}$ with law $\mathbb{P}^{k_1,k_2,a,b,\vec{x},\vec{y},f,g}_{\mathrm{avoid}}$ as $M\to\infty$. Since $C(\llbracket k_1,k_2\rrbracket\times[a,b])$ is a Polish space, see \cite[Lemma 2.2]{DFFS}, we may apply the Skorohod representation theorem, \cite[Theorem 6.7]{Billing}. We may move to a different probability space $(\Omega,\mathcal{F},\mathbb{P})$ that supports $C(\llbracket k_1,k_2\rrbracket\times[a,b])$-valued random variables with the laws of $\mathcal{Q}^M|_{\llbracket k_1,k_2\rrbracket\times[a,b]}$ and $\mathcal{Q}$, which we denote again by $\mathcal{Q}^M$ and $\mathcal{Q}$ for convenience, such that $\mathcal{Q}^M \to \mathcal{Q}$ uniformly $\mathbb{P}$-a.s. In the following we write $\mathbb{E}_M$ for expectation with respect to the law of $Z^M$ and $\mathbb{E}$ for expectation with respect to $\mathbb{P}$.
		
		By Definition \ref{walktilt}, for any bounded continuous $F : C(\llbracket k_1,k_2\rrbracket\times[a,b])\to\mathbb{R}$ we can write
		\begin{equation}\label{tiltRN}
			\mathbb{E}_M \left[F(\mathcal{Z}^M)\right] = \frac{\mathbb{E}\left[F(\mathcal{Q}^M) \exp \left(-\sum_{i=k_1}^{k_2} \frac{1}{M}\sum_{j=a_M}^{b_M-1} \mathcal{Q}^M_i(j)\right)\right]}{\mathbb{E}\left[\exp\left(-\sum_{i=k_1}^{k_2} \frac{1}{M}\sum_{j=a_M}^{b_M-1} \mathcal{Q}^M_i(j)\right)\right]}.
		\end{equation}
		Note that we can rewrite the area tilt in the numerator and denominator as
		\[
		\exp\left(-\sum_{i=k_1}^{k_2} \left( \frac{x_{M,i}}{M} + M^{-2/3}\sum_{j=a_M+1}^{b_M-1} \mathcal{Q}^M_i(jM^{-2/3}) \right)\right).
		\]
		The first term inside the outer sum goes to 0 as $M\to\infty$, since $M^{-1/3}\vec{x}_M$ converges by assumption. For the inner sum, note that $(a_M+1)M^{-2/3}, (b_M -1)M^{-2/3} \in [a,b]$, and using the uniform convergence of $\mathcal{Q}_i^M$ to $\mathcal{Q}_i$ on $[a,b]$ we can then estimate
		\begin{align*}
			\left|M^{-2/3}\sum_{j = a_M + 1}^{b_M - 1} \mathcal{Q}_i^M(jM^{-2/3}) - \int_a^b \mathcal{Q}_i(t)\,\mathrm{d}t\right| &\leq M^{-2/3}\sum_{j = a_M+1}^{b_M-1} \left|\mathcal{Q}_i^M(jM^{-2/3}) - \mathcal{Q}_i(jM^{-2/3})\right|\\
			&\qquad + \left| M^{-2/3}\sum_{j = a_M + 1}^{b_M - 1} \mathcal{Q}_i(jM^{-2/3}) - \int_a^b \mathcal{Q}_i(t)\,\mathrm{d}t\right|\\
			&\leq (b-a+2) \left\lVert \mathcal{Q}_i^M - \mathcal{Q}_i\right\rVert_\infty + o_M(1)\\
			&= o_M(1).
		\end{align*}
		Here $\lVert\cdot\rVert_\infty$ denotes the sup norm, and $o_M(1)$ denotes a quantity that tends to $0$ almost surely as $M\to\infty$. We have observed that the expression in the second line is $o_M(1)$ since the sum is a Riemann sum for the continuous function $\mathcal{Q}_i$ over $[a,b]$ and therefore converges to the integral.
		
		The last observation implies for any bounded continuous $G: C(\llbracket k_1,k_2\rrbracket\times[a,b])\to\mathbb{R}$ that $\mathbb{P}$-a.s.
		$$ \lim_{M \rightarrow \infty}G(\mathcal{Q}^M) \exp \left(-\sum_{i=k_1}^{k_2} \frac{1}{M}\sum_{j=a_M}^{b_M-1} Q^M_i(j)\right)  = G(\mathcal{Q}) \exp \left(-\sum_{i=k_1}^{k_2} \int_a^b \mathcal{Q}_i(t)\,\mathrm{d}t\right).$$
		Using this with $G = F$ and $G = 1$ in (\ref{tiltRN}), the bounded convergence theorem gives
		\begin{equation*}
			\begin{split}
				&\lim_{M \rightarrow \infty} \mathbb{E}_M \left[F(\mathcal{Z}^M)\right] = \frac{\mathbb{E}\left[F(\mathcal{Q}) \exp \left(-\sum_{i=k_1}^{k_2} \int_a^b \mathcal{Q}_i(t)\,\mathrm{d}t\right)\right]}{\mathbb{E}\left[\exp\left(-\sum_{i=k_1}^{k_2} \int_a^b \mathcal{Q}_i(t)\,\mathrm{d}t \right)\right]} \\
				&\qquad = \frac{\mathbb{E}\left[F(\mathcal{Q}) \exp \left(-\sum_{i=k_1}^{k_2} \int_a^b (\mathcal{Q}_i(t)-g(t))\,\mathrm{d}t\right)\right]}{\mathbb{E}\left[\exp\left(-\sum_{i=k_1}^{k_2} \int_a^b (\mathcal{Q}_i(t)-g(t))\,\mathrm{d}t \right)\right]},
			\end{split}
		\end{equation*}
		where in the last equality we are free to subtract $g$ in the numerator and denominator as it divides out. The second line is the expectation of the functional $F$ under $\mathbb{P}^{k_1,k_2,a,b,\vec{x},\vec{y},f,g}_{\mathrm{avoid},\vec{A}}$ with $A_i = 1$ for $i\in \llbracket k_1,k_2\rrbracket $ by Definition \ref{DefWeightedLaw}, and the proof is complete.
		
	\end{proof}

	%----------------------------------------------------------------------------------------------------------------------------------------------------------------------------------------------------------------------------
	%
	%     Section 3.2
	%
	%-------------------------------------------------------------------------------------------------------------------------------------------------------------------------------------------------------------------
	
	\subsection{Gibbs property for $\mathcal{X}^N$}\label{sec3.2} In this section we show that the Dyson Ferrari--Spohn diffusions satisfy the partial Brownian Gibbs property with respect to area tilts of Definition \ref{DefPBGP}. Our argument relies on the following result, which is a special case of \cite[Theorem A]{IVW}.
	
	\begin{proposition}\label{ConvDFS}
		Fix any $N\in\mathbb{N}$ and for $M \in \mathbb{N}$ define $\vec{z}_M = (M^{1/3},2M^{1/3},\dots,NM^{1/3})$. Let $Y^M$ have law $\mathbb{P}^{1,N,-M,M,\vec{z}_M,\vec{z}_M,\infty,0}_{\mathrm{avoid,walk},1/M}$ as in Definition \ref{walktilt} $($here $f(x) = \infty$ and $g(x) = 0$ for $x \in \llbracket -M , M \rrbracket)$. Viewing as usual $Y^M$ as a $\llbracket 1, N \rrbracket$-indexed line ensemble on $[-M, M]$, we define the $C(\llbracket 1, N \rrbracket \times [-M^{1/3}, M^{1/3}])$-valued random variable $\mathcal{Y}^M$ via
		\begin{equation}\label{rescaled}
			\mathcal{Y}^M(t) = M^{-1/3}Y^M(tM^{2/3}).
		\end{equation}
		Then $\mathcal{Y}^M$ converges as $M\to\infty$, in the topology of uniform convergence on compact sets, to a $\llbracket 1, N \rrbracket$-indexed line ensemble $\mathcal{Y}$ on $\mathbb{R}$, where $\mathcal{Y}_i(t) = 2^{-1/3}  \mathcal{X}^N_i(2^{2/3}t)$ for $i \in \llbracket 1, N \rrbracket$ and $\mathcal{X}^N$ is the Dyson Ferrari--Spohn diffusion of Section \ref{sec1.2}.
	\end{proposition}
	\begin{proof} From \cite[Theorem A]{IVW} we know that $\mathcal{Y}^M$ converges as $M\rightarrow \infty$ to the stationary diffusion $\mathcal{Y}$ on $\mathbb{R}$, whose fixed-time distribution has density proportional to $\Delta^{\mathcal{Y}}(\vec{x})^2 \cdot {\bf 1} \{ x_1 > \cdots > x_N > 0\}$, and whose infinitesimal generator is given by
		\begin{equation}\label{B1}
			L^{\mathcal{Y}} = \sum_{k=1}^N \left(\frac{1}{2}\frac{\partial^2}{\partial x_k^2} + \frac{\partial \log \Delta^{\mathcal{Y}}(\vec{x})}{ \partial x_k} \cdot \frac{\partial}{\partial x_k}\right).
		\end{equation}
		Here $\Delta^{\mathcal{Y}}(\vec{x}) = \det[\varphi^{\mathcal{Y}}_i(x_{N-j+1})]_{i,j=1}^N$, and $\{ \varphi^{\mathcal{Y}}_i\}_{i = 1}^\infty$ is the complete orthonormal family of eigenfunctions on $L^2(\mathbb{R}_+)$ with zero boundary conditions at $0$ for the Sturm-Liouville operator 
		\begin{equation}\label{B2}
			\mathsf{L}^{1} = \frac{1}{2} \frac{\operatorname{d}^2}{\mathrm{d}r^2} - q^{1}(r), \mbox{ where for $\alpha > 0$, }\, \mathsf{L}^{\alpha} = \frac{1}{2} \frac{\operatorname{d}^2}{\mathrm{d} r^2} - q^{\alpha }(r) \mbox{ with } q^{\alpha}(r) =  \alpha r.
		\end{equation}
		We mention that the indices of the $x$ variables are reversed compared to \cite[Theorem A]{IVW} as our convention is that $x_1 > x_2 > \cdots > x_N > 0,$ while \cite{IVW} assumes $x_N > \cdots > x_1 > 0$.
		
		For $\alpha > 0$ the complete orthonormal family of eigenfunctions of $\mathsf{L}^{\alpha}$ on $L^2(\mathbb{R}_+)$ with zero boundary conditions is given by
		\begin{equation}\label{B3}
			\{ \phi^{\alpha}_n \}_{n=1}^\infty \mbox{, where } \phi_{n}^{\alpha}(x) = c(n,\alpha) \cdot \mathrm{Ai}((2\alpha)^{1/3} x - \omega_n) \mbox{ and } \mathsf{L}^{\alpha} \phi_n^{\alpha} = -\omega_n  \alpha^{2/3} 2^{-1/3} \cdot \phi_n^{\alpha}. 
		\end{equation}
		As in Section \ref{sec1.2}, $\omega_n$ are the zeros of the Airy function $\mathrm{Ai}$ and $c(n,\alpha)$ are constants such that $\int_{0}^{\infty} (\phi_n^{\alpha}(x))^2\,\mathrm{d}x = 1$ for all $n \in \mathbb{N}$.
		
		Let us briefly explain the origin of (\ref{B3}). We first note that by a simple change of variables if $\{ \phi^{1/2}_n \}_{n=1}^\infty$ is the complete orthonormal family of eigenfunctions of $\mathsf{L}^{1/2}$, then $\phi^{\alpha}_n(x) = (2\alpha)^{1/6} \cdot \phi_n^{1/2}((2\alpha)^{1/3} x)$ for $n \in \mathbb{N}$ is the one for $\mathsf{L}^{\alpha}$. Thus it suffices to check (\ref{B3}) when $\alpha = 1/2$. In this case, the $\phi_n^{1/2}$ and their corresponding eigenvalues were computed in \cite[Section 4.12]{Titchmarsh46}. We mention that the formulas in \cite[Section 4.12]{Titchmarsh46} are in terms of $J_{\pm 1/3}$ (fractional Bessel functions of the first kind) and to obtain the formula in (\ref{B3}) one needs to apply \cite[10.4.15]{Abramowitz}.\\
		
		What remains is to show that the process $\mathcal{Y}$ has the same distribution as $\mathcal{Z}$, where $\mathcal{Z}_i(t) = 2^{-1/3}  \mathcal{X}^N_i(2^{2/3}t)$ for $i \in \llbracket 1, N \rrbracket$. Since in (\ref{B1}) we have the logarithmic derivative of a determinant involving $\phi_{n}^{1}$, we can rescale the functions in that determinant by arbitrary constants without affecting $L^{\mathcal{Y}}$. In particular, we set $\varphi_n^{\mathcal{Y}}(x) = \mathrm{Ai}(2^{1/3} x - \omega_n)$ for $n \in \mathbb{N}$. Combining the latter with (\ref{S1GEN}), we see by direct computation that $\mathcal{Z}$ has infinitesimal generator $L^{\mathcal{Y}}$ as in (\ref{B1}). We also see from the fixed-time density formula for $\mathcal{X}^N$ above (\ref{S1GEN}), that $\mathcal{Z}$ has the same fixed-time density as $\mathcal{Y}$ and so the two processes have the same distribution.
	\end{proof}
	We now turn to the main result of the section.
	\begin{proposition}\label{DFSGibbs}
		Fix $N\in\mathbb{N}$, and define the $\llbracket 1, N+1 \rrbracket$-indexed line ensemble $\mathcal{L}^N$ on $\mathbb{R}$ via
		$$\mathcal{L}^N_i(t) = \mathcal{X}^N_i(t) \mbox{ for $i = 1, \dots, N$, $t\in\mathbb{R}$, and } \mathcal{L}^N_{N+1}(t) = 0 \mbox{ for $t \in \mathbb{R}$},$$
		where $\mathcal{X}^N = (\mathcal{X}^N_1,\dots,\mathcal{X}^N_N)$ is the Dyson Ferrari--Spohn diffusion. Then, $\mathcal{L}^N$ satisfies the partial Brownian Gibbs property with respect to $\vec{A}$-area tilts, where $A_i = 1/2$ for $i\in \llbracket 1,N\rrbracket$. 
	\end{proposition}
	
	\begin{proof} Let $\mathcal{Y}$ be the $\llbracket 1, N \rrbracket$-indexed line ensemble $\mathcal{Y}$ on $\mathbb{R}$ in Proposition \ref{ConvDFS}, extended to an $\llbracket 1, N +1 \rrbracket$-indexed line ensemble by setting $\mathcal{Y}_{N+1}(t) = 0$ for $t \in \mathbb{R}$. By Lemma \ref{ScalingAGP}, it suffices to show that $\mathcal{Y}$ satisfies the partial Brownian Gibbs property with respect to $\vec{A}$-area tilts, where $A_i = 1$.
		
		The idea of the proof is simply to use Proposition \ref{ConvDFS}, which shows that $\mathcal{Y}$ is a scaling limit of non-intersecting random walks with area tilts, for which the Gibbs property with respect to discrete area tilts holds. Then the proof is just a matter of checking that this property indeed passes to the Brownian property in the limit, following the proof of \cite[Lemma 2.13]{DM}. 
		
		We must show that $\mathcal{Y}$ is $\mathbb{P}$-a.s. non-intersecting, and satisfies the following. Fix a set $K = \llbracket k_1, k_2\rrbracket \subseteq \llbracket 1,N\rrbracket$, $a,b\in\mathbb{R}$ with $a<b$, and a bounded Borel-measurable function $F : C(K\times[a,b])\to\mathbb{R}$. Then
		\begin{equation}\label{fsbgpF}
			\mathbb{E}\left[F(\mathcal{Y}|_{K\times[a,b]}) \mid \mathcal{F}_{\mathrm{ext}}(K\times(a,b))\right] = \mathbb{E}^{k_1,k_2,a,b,\vec{x},\vec{y},f,g}_{\mathrm{avoid},\vec{A}}[F(\mathcal{Q})],
		\end{equation} 
		where $\vec{x} = (\mathcal{Y}_{k_1}(a),\dots,\mathcal{Y}_{k_2}(a))$, $\vec{y} = (\mathcal{Y}_{k_1}(b),\dots,\mathcal{Y}_{k_2}(b))$, $f = \mathcal{Y}_{k_1-1}[a,b]$ (with $\mathcal{Y}_0 = \infty$), $g = \mathcal{Y}_{k_2+1}[a,b]$, and $\mathcal{Q}$ has law $\mathbb{P}_{\mathrm{avoid},\vec{A}}^{k_1,k_2,a,b,\vec{x},\vec{y},f,g}$. Note that the right hand side in (\ref{fsbgpF}) is well-defined since for any fixed $t\in\mathbb{R}$, the vector $(\mathcal{Y}_1(t), \dots, \mathcal{Y}_{N}(t)) = 2^{-1/3}(\mathcal{X}^N_1(2^{2/3}t),\dots,\mathcal{X}^N_N(2^{1/3}t))$ almost surely lies in $W_{N}^\circ \cap (0,\infty)^{N}$ by the definition in \cite[Section 2.2]{IVW} or alternatively the density formula \cite[(3.12)]{FS}. Note this does not immediately imply that $\mathcal{Y}$ is non-intersecting on the whole line.
		
		However, it does in fact follow from \eqref{fsbgpF} that $\mathcal{Y}$ is $\mathbb{P}$-a.s. non-intersecting. Indeed, take $k_1=1$, $k_2=N$, and $F(f_1,\dots,f_{N}) = \mathbf{1}\{f_1(s) > \cdots > f_{N}(s) > 0 \mbox{ for all }s\in[a,b]\}$. Then the right hand side of \eqref{fsbgpF} is $1$ by definition since $g = 0$, and taking expectation on the left hand side then implies that $(\mathcal{Y}_1,\dots,\mathcal{Y}_{N+1})$ is $\mathbb{P}$-a.s. non-intersecting on $[a,b]$. Taking the countable intersection over integers $a<b$ proves the claim.\\
		
		It remains to prove \eqref{fsbgpF}. By a standard monotone class argument (see Step 1 in the proof of \cite[Lemma 2.13]{DM}), it suffices to verify that if we fix $m\in\mathbb{N}$, $n_1,\dots,n_m\in\llbracket 1,N\rrbracket$, $t_1,\dots,t_m\in\mathbb{R}$, and bounded continuous functions $h_1,\dots,h_m : \mathbb{R}\to\mathbb{R}$, then setting $S = \{i\in\llbracket 1,m\rrbracket : n_i \in K, t_i\in [a,b]\}$ we have
		\begin{equation}\label{fsbgpmain}
			\mathbb{E} \left[ \prod_{i=1}^m h_i(\mathcal{Y}_{n_i}(t_i)) \right] = \mathbb{E} \left[\prod_{i\notin S} h_i(\mathcal{Y}_{n_i}(t_i)) \cdot \mathbb{E}_{\mathrm{avoid},\vec{A}}^{k_1,k_2,a,b,\vec{x},\vec{y},f,g} \left[ \prod_{j\in S} h_j(\mathcal{Q}_{n_j}(t_j)) \right]\right].
		\end{equation}
		
		Let $Y^M$ and $\mathcal{Y}^M$ be as in Proposition \ref{ConvDFS}. Note that the uniform over compacts convergence of $\mathcal{Y}^M$ to $\mathcal{Y}$ implies
		\begin{equation}\label{fsbgplim}
			\mathbb{E} \left[ \prod_{i=1}^m h_i(\mathcal{Y}_{n_i}(t_i)) \right]  = \lim_{M\to\infty} \mathbb{E} \left[ \prod_{i=1}^m h_i(\mathcal{Y}^M_{n_i}(t_i)) \right].
		\end{equation}
		
		Now let $a_M = \lfloor a M^{2/3}\rfloor$, $b_M = \lceil b M^{2/3}\rceil $, $\vec{x}_M = Y^M(a_M)$, $\vec{y}_M = Y^M(b_M)$, $f_M = Y^M_{k_1-1}\llbracket a_M,b_M\rrbracket$, and $g_M = Y^M_{k_2+1}\llbracket a_M,b_M\rrbracket$. By the Skorohod representation theorem, we may assume without loss of generality that $\mathcal{Y}^M \to \mathcal{Y}$ uniformly on $[a,b]$ as $M\to\infty$, $\mathbb{P}$-a.s. This implies in particular that $M^{-1/3}\vec{x}_M\to \vec{x}$, $M^{-1/3}\vec{y}_M\to \vec{y}$, and $M^{-1/3}f_M(tM^{2/3})\to f(t)$, $M^{-1/3}g_M(tM^{2/3})\to g(t)$ uniformly in $t\in [a,b]$, $\mathbb{P}$-a.s. Moreover, since $a_M M^{-2/3}\to a$, $b_M M^{-2/3}\to b$, we may assume $M$ is sufficiently large so that $t_i M^{2/3} \notin [a_M, b_M]$ for all $i\notin S$ with $n_i\in K$. 
		
		Now let $Z^M$ have law $\mathbb{P}_{\mathrm{avoid},\mathrm{walk},1/M}^{k_1,k_2,a_M,b_M,\vec{x}_M,\vec{y}_M,f_M,g_M}$, and define $\mathcal{Z}^M(t) = M^{-1/3} Z^M(tM^{2/3})$ for $t\in[a,b]$. By the Gibbs property with respect to discrete area tilts for $Y^M$ (see Remark \ref{discgibbs})
		\[
		\mathbb{E} \left[ \prod_{i=1}^m h_i(\mathcal{Y}^M_{n_i}(t_i)) \right] = \mathbb{E} \left[\prod_{i\notin S} h_i(\mathcal{Y}^M_{n_i}(t_i)) \cdot \mathbb{E}_{\mathrm{avoid},\mathrm{walk},1/M}^{k_1,k_2,a_M,b_M,\vec{x}_M,\vec{y}_M,f_M,g_M}\left[\prod_{j\in S} h_i(\mathcal{Z}^M_{n_j}(t_j))\right]\right].
		\]
		By Lemma \ref{walkconv}, the law of $\mathcal{Z}^M$ converges weakly to $\mathbb{P}^{k_1,k_2,a,b,\vec{x},\vec{y},f,g}_{\mathrm{avoid},\vec{A}}$ as $M\to\infty$, with $A_i = 1$ for all $i \in \llbracket k_1, k_2 \rrbracket$, and so we have $\mathbb{P}$-a.s.
		$$\lim_{M \rightarrow \infty} \mathbb{E}_{\mathrm{avoid},\mathrm{walk},1/M}^{k_1,k_2,a_M,b_M,\vec{x}_M,\vec{y}_M,f_M,g_M}\left[\prod_{j\in S} h_i(\mathcal{Z}^M_{n_j}(t_j))\right] = \mathbb{E}_{\mathrm{avoid},\vec{A}}^{k_1,k_2,a,b,\vec{x},\vec{y},f,g} \left[ \prod_{j\in S} h_j(\mathcal{Q}_{n_j}(t_j)) \right].$$
		From the last two equations and the bounded convergence theorem we conclude
		\[
		\lim_{M\to\infty}\mathbb{E} \left[ \prod_{i=1}^m h_i(\mathcal{Y}^M_{n_i}(t_i)) \right]  = \mathbb{E}\left[\prod_{i\notin S} h_i(\mathcal{Y}_{n_i}(t_i)) \cdot \mathbb{E}_{\mathrm{avoid},\vec{A}}^{k_1,k_2,a,b,\vec{x},\vec{y},f,g} \left[ \prod_{j\in S} h_j(\mathcal{Q}_{n_j}(t_j)) \right]\right].
		\]
		In combination with \eqref{fsbgplim}, this proves \eqref{fsbgpmain}.
		
	\end{proof}

	%----------------------------------------------------------------------------------------------------------------------------------------------------------------------------------------------------------------------------
	%
	%     Section 4
	%
	%-------------------------------------------------------------------------------------------------------------------------------------------------------------------------------------------------------------------
	
	\section{Proof of Theorem \ref{main}}\label{sec4}
	
	In this section we prove Theorem \ref{main}; we split the argument into two steps. In the first step, which is the content of Section \ref{sec4.1}, we show that $\widetilde{\mathcal{X}}^N$ converges to the Airy line ensemble in the sense of finite-dimensional distributions as $N \rightarrow \infty$. In the second step, which is the content of Section \ref{sec4.2}, we improve the convergence to that of uniform convergence over compact sets by utilizing the Gibbs property enjoyed by $\mathcal{X}^N$ in view of Proposition \ref{DFSGibbs}.
	
	%----------------------------------------------------------------------------------------------------------------------------------------------------------------------------------------------------------------------------
	%
	%     Section 4.1
	%
	%-------------------------------------------------------------------------------------------------------------------------------------------------------------------------------------------------------------------
	
	\subsection{Finite-dimensional convergence}\label{sec4.1} Let us fix $t_1, \dots, t_m \in \mathbb{R}$ such that $t_1 < t_2 < \cdots < t_m.$ From \cite[Section 3.3]{FS} we know that the point process 
	\begin{equation*}
		\eta_N := \sum_{n = 1}^m \sum_{k = 1}^N \delta_{(\mathcal{X}^N_k(t_n), t_n)}
	\end{equation*}
	is a determinantal point process on the space $(0,\infty) \times \{t_1, \dots, t_m\}$. Here, the measure on $(0, \infty) \times \{t_1, \dots, t_m\}$ is the product of the Lebesgue measure on $(0, \infty)$ and the counting measure on $\{t_1, \dots,t_m\} $, and the correlation kernel is given by
	\begin{equation}\label{kernel}
		K_N(x_i,t_i; x_j,t_j) = \begin{dcases} \sum_{k=1}^N e^{-\frac{1}{2}\omega_k (t_j-t_i)} \varphi_k(x_i)\varphi_k(x_j), & t_i\geq t_j,\\
			-\sum_{k=N+1}^\infty e^{-\frac{1}{2}\omega_k (t_j-t_i)} \varphi_k(x_i)\varphi_k(x_j), & t_i < t_j,
		\end{dcases}
	\end{equation}
	where $\varphi_k$ are given by
	$$\varphi_k(x) = \frac{\mathrm{Ai}(x-\omega_k)}{(-1)^{k-1} \mathrm{Ai}'(-\omega_k)} \mbox{ for $k \geq 1$,}$$
	and $-\omega_1 > -\omega_2 > \cdots$ are the zeros of the Airy function $\mathrm{Ai}$. For a comprehensive background on determinantal point processes we refer the reader to \cite{DetPP}.
	
	Using a simple change of variables and conjugation, see \cite[(3.31)]{FS}, the point process 
	\begin{equation}\label{DPP}
		\tilde{\eta}_N := \sum_{n = 1}^m \sum_{k = 1}^N \delta_{(\widetilde{\mathcal{X}}^N_k(t_n), t_n)}
	\end{equation}
	is a determinantal point process on the space $(-c_1 N^{2/3}, \infty) \times \{t_1, \dots, t_n\}$, with the product of Lebesgue and counting measure, and correlation kernel
	\begin{equation}\label{Scaledkernel}
		\tilde{K}_N(\xi_i,\tau_i; \xi_j,\tau_j) = e^{(\tau_j - \tau_i) c_1 N^{2/3}} K_N(c_1 N^{2/3} + \xi_i, 2\tau_i; c_1 N^{2/3} + \xi_j, 2\tau_j).
	\end{equation}
	
	We also have from the definition of the Airy line ensemble $\mathcal{A}^{\mathrm{Airy}}$ in \cite[Theorem 3.1]{CH14} that the point process 
	\begin{equation}\label{APP}
		\eta_{\mathrm{Airy}} := \sum_{n = 1}^m \sum_{k = 1}^\infty \delta_{(\mathcal{A}^{\mathrm{Airy}}_k(t_n), t_n)}
	\end{equation}
	is a determinantal point process on $\mathbb{R} \times \{t_1, \dots, t_m\}$, with the product of Lebesgue and counting measure, and correlation kernel given by the {\em extended Airy kernel}
	\begin{equation}\label{AiryKernel}
		A(\tau_i, \xi_i; \tau_j, \xi_j) = \begin{cases}\int_0^\infty e^{-\lambda (\tau_i - \tau_j) }\mathrm{Ai}(\xi_i + \lambda) \mathrm{Ai}(\xi_j + \lambda) \,\mathrm{d}\lambda, &\tau_i \geq \tau_j,\\ -\int_{-\infty}^0 e^{-\lambda (\tau_i - \tau_j) }\mathrm{Ai}(\xi + \lambda) \mathrm{Ai}(\xi_j + \lambda) \,\mathrm{d}\lambda, &\tau_i < \tau_j.\end{cases}
	\end{equation}
	The extended Airy kernel was introduced by \cite{Spohn} in the context of the polynuclear growth model.
	
	We require the following result, which is a direct consequence of \cite[Propositions 3.3 and 3.5]{FS}.
	\begin{lemma} For any $s,t \in \mathbb{R}$ and $a \in \mathbb{N}$ we have 
		\begin{equation}\label{kernelConv}
			\lim_{N \rightarrow \infty} \sup_{x,y \in [-a, a]} \big|\tilde{K}_N(x,s; y,t) - A(x,s;y,t) \big| = 0.    
		\end{equation}
		For any $t \in \mathbb{R}$ and $b \in \mathbb{R}$, there exist constants $N_1, C_1 > 0$ such that 
		\begin{equation}\label{kernelTail}
			\big|\tilde{K}_N(x,t; y,t) \big| \leq C_1  e^{-x - y} \mbox{ if } x,y \geq b \mbox{ and } N \geq N_1.
		\end{equation}
	\end{lemma}
	
	We now turn to the main result of the section.
	\begin{proposition}\label{FDCProp} Assume the same notation as in Theorem \ref{main}. Then, $\widetilde{\mathcal{X}}^N$ converges weakly as $N\to\infty$ to the Airy line ensemble $\mathcal{A}^{\mathrm{Airy}}$ in the sense of finite-dimensional distributions.
	\end{proposition}
	\begin{proof} The proof we present is a straightforward adaptation of the argument used in \cite{DNV} to prove finite-dimensional convergence in Theorem 1.5 within the same paper. As such, we will be brief and only emphasize the straightforward modifications that need to be made.
		
		Let us define the random measures on $\mathbb{R}$
		$$P^N_j = \sum_{k = 1}^N \delta_{\widetilde{\mathcal{X}}^N_k(t_j)} \mbox{ for } j = 1, \dots, m.$$
		As explained in \cite[pp.\! 16-17]{DNV}, to conclude the statement of the proposition it suffices to show that 
		\begin{enumerate}
			\item The measures $\{P_j^N\}_{j = 1}^m$ converge jointly in distribution with respect to the vague topology to the random measures $P_j^{\mathrm{Airy}} = \sum_{k = 1}^{\infty} \delta_{\mathcal{A}^{\mathrm{Airy}}_k(t_j)}$, $j = 1, \dots, m$ as $N \rightarrow \infty$. 
			\item For each $j = 1, \dots, m$ we have 
			\begin{equation}\label{tailProb}
				\lim_{M \rightarrow \infty} \limsup_{N \rightarrow \infty} \mathbb{E} \left[ P_j^N[M,\infty) \right] = 0.
			\end{equation}
		\end{enumerate}
		For background on the vague topology we refer the reader to \cite[Chapter 4]{Kallenberg}.
		
		For establishing the first point above we can repeat verbatim the arguments in \cite{DNV}\textemdash we only need to replace the measures $\nu_{N,i}$ in that paper with the Lebesgue measure on $(-c_1 N^{2/3}, \infty).$ The essential ingredients we need are that $\nu_{N,i}$ converge vaguely to the Lebesgue measure on $\mathbb{R}$ (which is clear in our case) and that $\tilde{K}_N(x,s; y,t)$ converges for fixed $s,t$ uniformly over compact sets (in $x,y$) to $A(x,s;y,t)$, which holds from (\ref{kernelConv}).
		
		For the second point, we note that for each $M \in \mathbb{R}$, $j = 1, \dots, m$ and all large $N$
		$$\mathbb{E} \left[ P_j^N[M,\infty) \right] = \int_{M}^\infty \tilde{K}_N(x,t_j; x,t_j) \,\mathrm{d}x \leq (C_1/2)e^{-2M},$$
		where in the last inequality we used (\ref{kernelTail}). The latter implies (\ref{tailProb}) and the proposition is proved.
	\end{proof}
	
	%----------------------------------------------------------------------------------------------------------------------------------------------------------------------------------------------------------------------------
	%
	%     Section 4.2
	%
	%-------------------------------------------------------------------------------------------------------------------------------------------------------------------------------------------------------------------
	
	\subsection{Uniform convergence over compacts}\label{sec4.2} In this section we conclude the proof of Theorem \ref{main}. The main idea is to apply \cite[Theorem 3.8]{CH14}, which allows one to utilize the Brownian Gibbs property to improve finite-dimensional convergence to one that is uniform over compact sets.
	
	\begin{proof}[Proof of Theorem \ref{main}]
		
		Let $\mathcal{L}^N = \{\mathcal{L}_i^N\}_{i = 1}^{N+1}$ be the ensemble from Proposition \ref{DFSGibbs}. If we define 
		$$\hat{\mathcal{L}}^N_i(t) = 2^{-1/2} \mathcal{L}_i^N(2t) - c_1 2^{-1/2}  N^{2/3}  \mbox{ for $i = 1, \dots, N+1$ and $t \in \mathbb{R}$},$$
		then by Lemma \ref{ScalingAGP} we know that $\tilde{\mathcal{L}}^N$ satisfies the partial Brownian Gibbs property with respect to $\vec{A}$-area tilts with $A_i = 2^{1/2}$ for $i \in \llbracket 1, N+1 \rrbracket$ as an $\llbracket 1, N +1\rrbracket$-indexed line ensemble on $\mathbb{R}$. From Lemma \ref{GibbsPar} with $A=2^{1/2}$ we see that 
		\begin{equation}\label{ASD1}
			\begin{split}
				\tilde{\mathcal{L}}^N_i(t) := \hat{\mathcal{L}}^N_i(t) - 2^{-1/2} \, t^2 \mbox{ for } i = 1, \dots, N + 1 \mbox{ and } t \in \mathbb{R},
			\end{split}
		\end{equation}
		satisfies the partial Brownian Gibbs property, cf. \cite[Definition 2.7]{DM}, as an $\llbracket 1, N +1\rrbracket$-indexed line ensemble on $\mathbb{R}$. Using the definition of $\widetilde{\mathcal{X}}^N$ in (\ref{FSrescaled}) we have
		\begin{equation}\label{ASD2}
			\begin{split}
				\tilde{\mathcal{L}}^N_i(t) := 2^{-1/2} \left( \widetilde{\mathcal{X}}_i^N(t) - t^2 \right) \mbox{ for } i = 1, \dots, N \mbox{ and } t \in \mathbb{R}.
			\end{split}
		\end{equation}
		From Proposition \ref{FDCProp} we know that $\tilde{\mathcal{L}}^N$ converges in the finite-dimensional sense as $N \rightarrow \infty$ to the $\mathbb{N}$-indexed line ensemble $ 2^{-1/2} \big(\mathcal{A}^{\mathrm{Airy}}_i(t) - t^2 \big) $, which is precisely the parabolic Airy line ensemble $\mathcal{L}^{\mathrm{Airy}}$ from Section \ref{sec1.2}.
		
		The latter observation shows that $\tilde{\mathcal{L}}^N$ satisfy hypotheses $(H2)_{k,T}$ in \cite[Definition 3.3]{CH14}, and since $\mathcal{L}^{\mathrm{Airy}}$ is non-intersecting almost surely, we see that they also satisfy hypotheses $(H3)_{k,T}$ in the same definition. Finally, we mention that $\tilde{\mathcal{L}}^N$ satisfy hypotheses $(H1)_{k,T}$ in \cite[Definition 3.3]{CH14} if one replaces ``Brownian Gibbs property'' with ``partial Brownian Gibbs property.'' The distinction between the two is mild and has to do with the $(N+1)$-st curve in $\tilde{\mathcal{L}}^N$ being deterministic and not itself Brownian, see \cite[Remark 2.9]{DM} for a careful explanation. The important point is that the results in \cite[Section 3]{CH14} all hold if one replaces the ``Brownian Gibbs property'' with the ``partial Brownian Gibbs property'' as their argument never uses properties of the bottom curves in the ensembles. Overall, we see that the conditions of \cite[Theorem 3.8]{CH14} are satisfied and we conclude that $\tilde{\mathcal{L}}^N$ converge uniformly over compacts to $\mathcal{L}^{\mathrm{Airy}}$ as $N \rightarrow \infty$. We mention that the convergence statement is not in \cite[Theorem 3.8]{CH14} but is implied by \cite[Proposition 3.6]{CH14}. 
		
		The uniform over compacts convergence of $\tilde{\mathcal{L}}^N$ to $\mathcal{L}^{\mathrm{Airy}}$ as $N \rightarrow \infty$, equation (\ref{ASD2}), and the relation $\mathcal{L}^{\mathrm{Airy}}_i(t) = 2^{-1/2} \big(\mathcal{A}^{\mathrm{Airy}}_i(t) - t^2 \big)$ for $i \in \mathbb{N}$ together imply the statement of Theorem \ref{main}.
		
	\end{proof}

	\bibliographystyle{alpha}
	\bibliography{bib}

\end{document}